\documentclass[11pt]{amsart}
\usepackage{geometry}                
\usepackage{graphicx}
\usepackage{amssymb}
\usepackage{amsmath}
\usepackage{color}
\usepackage{hyperref}
\usepackage{pdfsync}
\usepackage{tikz}
 
\usepackage[all]{xy}
\xyoption{matrix}
\xyoption{arrow}
 
 \tikzset{help lines/.style={step=#1cm,very thin, color=gray},
help lines/.default=.5} 
\tikzset{thick grid/.style={step=#1cm,thick, color=gray},
thick grid/.default=1} 

\newtheorem{thm}{Theorem}[section]
\newtheorem{lem}[thm]{Lemma}
\newtheorem{cor}[thm]{Corollary}
\newtheorem{prop}[thm]{Proposition}

\theoremstyle{definition}
\newtheorem{defn}[thm]{Definition}
\newtheorem{eg}[thm]{Example}

\theoremstyle{remark}
\newtheorem{rem}[thm]{Remark}
 
\numberwithin{equation}{section}

\newcommand{\mat}[1]{\ensuremath{
\left[\begin{matrix}#1
\end{matrix}\right]
}}
 
\newcommand{\vs}[1]{\vskip .#1 cm} 

\newcommand{\xrarrow}{\xrightarrow} 
\newcommand{\xlarrow}{\xleftarrow}
\newcommand{\ot}{\leftarrow}

\newcommand{\into}{\hookrightarrow}
 \newcommand{\onto}{\twoheadrightarrow}

\newcommand{\smallcoprod}{\,{\textstyle{\coprod}}\,}

\def\<{\left<}
\def\>{\right>}

\DeclareMathOperator{\Hom}{Hom}%
\DeclareMathOperator{\Ext}{Ext}%
\DeclareMathOperator{\Ind}{Ind}%
\DeclareMathOperator{\Aut}{Aut}

\newcommand{\field}[1]{\mathbb{#1}}
\newcommand{\ZZ}{\ensuremath{{\field{Z}}}}

\newcommand{\RR}{\ensuremath{{\field{R}}}}

\newcommand{\commentout}[1]{}

\newcommand{\cA}{\ensuremath{{\mathcal{A}}}}

\newcommand{\cB}{\ensuremath{{\mathcal{B}}}}
\newcommand{\cC}{\ensuremath{{\mathcal{C}}}}
\newcommand{\cD}{\ensuremath{{\mathcal{D}}}}

\newcommand{\cG}{\ensuremath{{\mathcal{G}}}}
\newcommand{\cH}{\ensuremath{{\mathcal{H}}}}
\newcommand{\cI}{\ensuremath{{\mathcal{I}}}}

\newcommand{\cK}{\ensuremath{{\mathcal{K}}}}

\newcommand{\cN}{\ensuremath{{\mathcal{N}}}}
\newcommand{\cP}{\ensuremath{{\mathcal{P}}}}
\newcommand{\cQ}{\ensuremath{{\mathcal{Q}}}}
\newcommand{\cR}{\ensuremath{{\mathcal{R}}}}
\newcommand{\cS}{\ensuremath{{\mathcal{S}}}}

\newcommand{\cU}{\ensuremath{{\mathcal{U}}}}

\newcommand{\cW}{\ensuremath{{\mathcal{W}}}}

\newcommand\vare{\varepsilon}

\title{A category of noncrossing partitions}
\author{Kiyoshi Igusa}
\address{Department of Mathematics, Brandeis University, Waltham, MA 02454}\email{igusa@brandeis.edu}



\begin{document}

\begin{abstract} In \cite{IOTW4}, we introduced ``picture groups'' and computed the cohomology of the picture group of type $A_n$. This is the same group what was introduced by Loday \cite{Loday} where he called it the ``Stasheff group''. In this paper, we give an elementary combinatorial interpretation of the {\color{blue}``cluster morphism category'' constructed in \cite{IT13} in the special case of the linearly oriented quiver of type $A_n$.} We prove that the classifying space of this category is locally $CAT(0)$ and thus a $K(\pi,1)$. We prove a more general statement that classifying spaces of certain ``cubical categories'' are locally $CAT(0)$. The objects of our category are the classical noncrossing partitions introduced by Kreweras \cite{Kreweras}. The morphisms are binary forests. This paper is independent of \cite{IT13} and \cite{IOTW4} except in the last section where we use \cite{IT13} to compare our category with the category with the same name given by Hubery and Krause \cite{HK}.
\end{abstract}

\maketitle

\tableofcontents

\section*{Introduction}
The concept of ``pictures'' was introduced by the author in his PhD thesis \cite{Igthesis} to define the algebraic K-theory invariant for {\color{blue}the fundamental group} $\pi_1$ of the space of pseudoisotopies of a smooth manifold. They are combinatorially equivalent to well-known objects in topology called ``spherical diagrams'' which are also called ``diagrams of relations'' or ``Peiffer diagrams''. See \cite{Peiffer}, \cite{GeHo86}, \cite{Stallings87} for some of the earlier works on these diagrams. Later, in \cite{IOr}, pictures were used to prove the $k$-slice conjecture for Milnor's $\overline\mu$ link invariants. In \cite{IOTW2b}, the sequel of \cite{ITW}, propictures are introduced. Canonical propictures associated to quivers of type $\tilde{A_n}$ are constructed and shown to be related to periodic trees and cluster tilting objects {\color{blue}for the quiver $\tilde A_n$}. 

In \cite{IOTW4}, picture groups are introduced. These are the universal groups associated to the canonical semi-invariant picture of any modulated quiver of finite type. The special case of $A_n$ is studied. An explicit model for the classifying space of the associated picture group $G(A_n)$ is constructed by pasting together Stasheff associahedra. Using this model, which we call the ``picture space'', the integral cohomology of $G(A_n)$ is computed. It is shown to be free abelian in every degree with rank equal to the ``ballot numbers''. In the special case of $A_n$ with straight orientation $1\ot 2\ot \cdots \ot n$, the picture group was first studied by Loday \cite{Loday} who called it the ``Stasheff group''.

The paper \cite{IOTW4} uses the fact that the picture space $X(A_n)$ is a {\color{blue}model for the classifying space} $K(\pi,1)$ for the picture groups $G(A_n)$. This is proved in general in \cite{IT13}. The purpose of the present paper is to given an elementary proof of this basic fact in the special case of $A_n$ with straight orientation, the case considered in \cite{Loday}. The proof is based on the following theorem of Gromov \cite{Gromov87}.

\begin{thm}[Gromov]\label{thm of Gromov}
A simply connected cubical space is $CAT(0)$ if and only if the link of every vertex is a flag complex.
\end{thm}

Since $CAT(0)$ spaces are contractible, we obtain the following conclusion.

\begin{cor}
If a connected locally cubical space has the property that the link of every vertex is a flag complex then it is a $K(\pi,1)$.
\end{cor}

A \emph{cubical space} is space which is a union of cubes $I^n$ for $I=[0,1]$ so that the intersection of any two cubes is a common face. A space is \emph{locally cubical} if its universal covering space is cubical. A \emph{flag complex} is a simplicial complex with the property that every set of $n+1$ vertices which are pairwise connected by edges spans an $n$-simplex. {\color{blue}For example, any finite set of vertices with no edges is a flag complex. This implies that any graph is locally $CAT(0)$.} It is clear that the link of every vertex of a locally cubical complex is a simplicial complex. But it is not always a flag complex.

Noncrossing partitions were introduced by Kreweras \cite{Kreweras} and later generalized in \cite{BW}, \cite{Bessis}. The relationship to representation theory originates in \cite{Ingalls-Thomas}. In \cite{HK} generalized noncrossing partitions become objects of a category. In this paper we used the classical noncrossing partitions of a single ordered set. This is a simplification of a special case of a construction from \cite{IT13}. It is self-contained except for the last section. {\color{blue}Readers should also see Chapoton and Nadeau's short combinatorial reinterpretation of our category and several enumeration results with mysterious symmetries which we don't know how to explain \cite{CN}.}

The contents of this paper are as follows. In Section 1 we define noncrossing partitions on a finite totally ordered set $V$ with $n$ elements. These are the objects of the category $\cN\cP(n)$. We call two parts of a partition \emph{adjacent} if merging those parts produces another noncrossing partition. We consider trees having parts of a partition as vertices and edges connecting certain adjacent parts. 

{\color{blue}There exists at least one morphism $\cS\to \cR$ if each part of $\cR$ is a union of parts of $\cS$. Each morphism is given by a collection of binary trees, one for each part of $\cR$ which prescribes how its constituent parts in $\cS$ are assembled. Thus, a morphism $\cS\to \cR$ is a binary forest, with one component for each part of $\cR$.}

Section 2 is devoted to proving that composition of morphisms as given in Section 1 is well-defined and associative. In the language of Section 3, this depends on ``backward links'' of objects being flag complexes.

In Section 3 we define \emph{cubical categories}. These describe general properties of the category $\cN\cP(n)$ which insure that its classifying space $B\cN\cP(n)$ is (locally) cubical. The local properties of the category $\cN\cP(n)$, as proven in detail in Section 2, imply that the universal covering of $B\cN\cP(n)$ is $CAT(0)$ and therefore contractible.

In Section 4 we compute the fundamental group of $B\cN\cP(n)$ and show that it is equal to the picture group  defined in \cite{IOTW4}, namely the picture group for $A_{n-1}$ with straight orientation. This group was first considered by Loday \cite{Loday} who called it the ``Stasheff group'' since its $K(\pi,1)$ is a quotient space of a Stasheff associahedron. However, \cite{IOTW4} has been revised and the results extended to picture groups of $A_n$ with any orientation. In \cite{IT22}, we use \cite{ST} to extend this result to all other Dynkin quivers and some extended Dynkin diagrams. 

In Section 5, {\color{blue}we recall the definition of the cluster category of any finite dimension hereditary algebra. We also recall the definition of the cluster morphism category of an hereditary algebra. We use this to} compare the category of noncrossing partitions in this paper with the category of noncrossing partitions given by Hubery and Krause in \cite{HK} which we will call $\cH\cK$. Basically the statement is that there is a category {\color{blue}whose objects are all} cluster categories and a contravariant functor from that category to $\cH\cK$. The root system of the quiver $A_{n-1}$ with straight orientation gives an object in $\cH\cK$. Our category $\cN\cP(n)$ is equivalent to a full subcategory of the comma category over this single object of $\cH\cK$. 

\section{Noncrossing partitions and binary forests}

{\color{blue}
Before starting the technical discussion, we explain the key property of a noncrossing partition and the morphisms between them. This is \emph{pairwise compatibility}. A partition of any linearly ordered set is called ``noncrossing'' if it is pairwise noncrossing, i.e., no two parts ``cross''. See Figure \ref{fig: nonX is pairwise condition}. We think of creating a noncrossing partition as a dynamic process. If we merge $1$ and $5$ to form $A$ we cannot then merge $3$ with $6$. Referring to the right side of Figure \ref{fig: nonX is pairwise condition}, we consider which pairs of parts can be merged. This is a pairwise condition: any collection of pairs of parts can be merged at the same time provided that no \emph{two} of these pairs produces a crossing partition. These are the objects and morphisms of a category of noncrossing partitions. Both objects and morphisms are defined by pairwise compatibility conditions. We will see that morphisms can be encoded with binary trees and these also satisfy the key property: edges form a binary tree if they are pairwise compatible and maximal with this property.

The ``pairwise compatibility'' condition also describes ``partial clusters'' in a cluster category. These are collections of objects in the category so that no two extend each other. (See \cite{BMRRT} for details.) When the collection is maximal, it is called a ``cluster''. Representations of a quiver of type $A_{n}$: $1\ot  2\ot \cdots\ot n$ will form a cluster category if we add shifted projective objects $P_i[1]$. The ``cluster morphism category'' for a quiver of type $A_{n}$ (See \cite{IT13}) has partial cluster as objects and extensions of these partial clusters to larger partial clusters. The category of noncrossing partitions on $n+1$ parts exactly mirrors this structure. In fact the two categories are isomorphic. Thus, the category of noncrossing partitions on $n+1$ parts is an elementary combinatorial construction for the more abstract cluster morphism category on the cluster category of type $A_{n}$. For example the parts $A=\{0,4\}$ and $B=\{2,6\}$ correspond to the modules $P_4$ and $I_3$. The parts $A$,$B$ cross since the modules $P_4,I_3$ extend each other. The nontrivial extension is $P_4\to P_6\oplus M\to I_3$ where $M=P_4/P_2$ is the image of $P_4\to I_3$.
}

{
\begin{figure}[htbp]
\begin{center}
\begin{tikzpicture}[scale=.8]
{
	\foreach \x in {0.55,2.55,4.55}
	{\draw[thick, rounded corners=.35cm] (\x,-.35) rectangle +(.9,.75);
	} 
	\draw 
	(0,0) node{0}
	(0,.9) node{$A$}
	(1,0) node{1}
	(2,0) node{2}
	(2,1.8) node{$B$}
	(3,0) node{3}
	(4,0) node{4}
	(5,0) node{5}
	(6,0) node{6};

%
\begin{scope}[xshift=-1cm]
	\draw[red,thick,rounded corners=.35cm]  (2.6,2.5)--(2.6,-.4)--(3.4,-.4)--(3.4,1.8)--(6.6,1.8)--(6.6,-.4)--(7.4,-.4)--(7.4,2.5)--cycle;
\end{scope}
\begin{scope}[xshift=-1cm]
	\draw[blue,thick,rounded corners=.35cm]  (.6,1.2)--(.6,-.4)--(1.4,-.4)--(1.4,.7)--(4.6,.7)--(4.6,-.4)--(5.4,-.4)--(5.4,1.2)--cycle;
\end{scope}
}
{
\begin{scope}[xshift=8cm]
	\foreach \x in {0.55,1.55,2.55,4.55,5.55}
	{\draw[thick, rounded corners=.35cm] (\x,-.35) rectangle +(.9,.75);
	} 
	\draw 
	(0,0) node{0}
	(0,.9) node{$A$}
	(1,0) node{1}
	(2,0) node{2}
	(3,0) node{3}
	(4,0) node{4}
	(5,0) node{5}
	(6,0) node{6};
%
%
\begin{scope}[xshift=-1cm]
	\draw[blue,thick,rounded corners=.35cm]  (.6,1.2)--(.6,-.4)--(1.4,-.4)--(1.4,.7)--(4.6,.7)--(4.6,-.4)--(5.4,-.4)--(5.4,1.2)--cycle;
\end{scope}
\end{scope}
}

\end{tikzpicture}
\color{blue}
\caption{On the left, two parts cross: $A=\{0,4\}$ and $B=\{2,6\}$ cross since $0<2<4<6$ (equivalently, the closed intervals $[0,4],[2,6]$ overlap). If we take one of these, say $A$, on the right, we cannot merge $1,2$ or $3$ with $5$ or $6$. We can merge several pair on the right as long as these pairs are not crossing. For example we cannot merge $2$ with $A$ and $1$ with $3$. This pair of operations ``cross''.}
\label{fig: nonX is pairwise condition}
\end{center}
\end{figure}
}
The category of noncrossing partitions will be defined in several steps. In this section we define the objects and morphisms of the category. In the next section we {\color{blue}define a composition of morphisms and derive a formula for this composition law in order to prove that composition is associative.} The objects of our category will be noncrossing partitions of a finite totally ordered set. Its morphisms will be given by ``binary forests.''

\subsection{Noncrossing partitions}

A \emph{partition} of any set $V$ is a set $\cS$ whose elements are disjoint subsets $S_i\subseteq V$ with the property that $V=\coprod S_i$. The elements $S_i\in\cS$ are called the \emph{parts} of the partition $\cS$. Assuming that $V$ is finite, the difference between the cardinality of $V$ and that of $\cS$ will be called the \emph{rank} of the partition. 
\[
	rk\,\cS:=|V|-|\cS|
\]
For example, there is a unique partition of $V$ of rank 0 given by partitioning $V$ into singletons. We say that a partition $\cR$ is a \emph{refinement} of $\cS$ if every part of $\cR$ lies in a part of $\cS$. We say that $\cS$ is obtained from $\cR$ by \emph{merging} parts of $\cR$ together. It is clear that, if $\cR$ is a refinement of $\cS$ and the difference between their ranks is $k$ then $\cR$ can be obtained from $\cS$ in $k$ steps where, in each step, two parts of the partition are merged.

When $V$ is totally ordered, a partition of $V$ is called \emph{noncrossing} if there do not exist $a<b<c<d$ so that $a,c$ lie in one part and $b,d$ lie in another. This can also be described as follows. Assume $V$ is a finite subset of $\RR$. {\color{blue}We usually take $V=\{1,2,3,\cdots,n\}$.} Then the \emph{support} of any part $A\subseteq V$ is the closed interval ${supp\,A:=[a,a']}\subset\RR$ where $a,a'$ are the minimum and maximum elements of $A$, respectively. A partition $\cS$ of $V$ is noncrossing if, for any two parts $A,B\in\cS$, one of the two sets is disjoint from the support of the other. {\color{blue}For example, the subsets $A=\{2,4\}$ and $B=\{1,5,6\}$ of $V$ are noncrossing since the support of $A$ is the closed interval $[2,4]$ which is disjoint from the finite set $B$.}

An important easy observation is the following.

\begin{prop}\label{characterization of refinement}
Given any noncrossing partition $\cS$ of $V$, another noncrossing partition of $V$ can be given by taking the union of arbitrary noncrossing partitions of every part of $\cS$.
\end{prop}

\begin{proof}
Given any two parts of $\cS$ the support of one, say $A$, is disjoint from the other, say $B$. Then the support of any part of $A$ is disjoint from any part of $B$.
\end{proof}

This observation implies that it is very easy to determine how one part of a noncrossing partition can be split into two parts. However, the conditions for the converse operation are not immediate.

\begin{defn} {\color{blue} Two parts of a noncrossing partition $\cS$ of a totally ordered set $V$ will be called \emph{adjacent} if, when we replace them with their union, the resulting partition of $V$ is still noncrossing.}
\end{defn}

{
\begin{figure}[htbp]
\begin{center}
\begin{tikzpicture}
	\foreach \x in {1.55,3.55,5.55}
	{\draw[thick, rounded corners=.35cm] (\x,-.35) rectangle +(.9,.75);
	} 
	\draw (0,0) node{1}
	(1,0) node{2}
	(2,0) node{$C$:3}
	(3,0) node{4}
	(4,0) node{$D$:5}
	(5,0) node{6}
	(6,0) node{$E$:7}
	(7,0) node{8};
	\draw[thick,rounded corners=.35cm]  (-.4,2)--(-.4,-.4)--(.4,-.4)--(.4,1.5)--(4.6,1.5)--(4.6,-.4)--(5.4,-.4)--(5.4,1.5)--(6.6,1.5)--(6.6,-.4)--(7.4,-.4)--(7.4,2)--cycle;
	\draw (0,1.5) node{$A$};
	\draw (1,.8) node{$B$};
	\draw[thick,rounded corners=.35cm]  (.6,1.2)--(.6,-.4)--(1.4,-.4)--(1.4,.7)--(2.6,.7)--(2.6,-.4)--(3.4,-.4)--(3.4,1.2)--cycle;
\end{tikzpicture}
\color{blue}
\caption{Example of noncrossing partition of $V=\{1,2,3,4,5,6,7,8\}$ into five parts $A=\{1,6,8\},B=\{2,4\},C=\{3\}, D=\{5\}, E=\{7\}$.}
\label{Fig: first example of noncrossing partition}
\end{center}
\end{figure}
}

\begin{eg}\label{eg: 1.3}
Take the partition of $V=\{1,2,3,4,5,6,7,8\}$ into 5 parts as shown in Figure \ref{Fig: first example of noncrossing partition}.  This is a noncrossing partition.

Of the $\binom52=10$ pairs of parts, five are adjacent and five are not:
\begin{enumerate}
\item $E$ is not adjacent to $B,C$ or $D$. The reason is that the union of $E$ with $B,C$ or $D$ would contain $6$ in its support and would thus cross $A$.
\item $C,D$ are not adjacent since $C\cup D$ crosses $B$.
\item $A,C$ are not adjacent since $A\cup C$ crosses $B$.
\item $B,D$ are adjacent.
\item The other four pairs are adjacent and ordered: $(B,A)$, $(D,A)$, $(C,B)$, $(E,A)$. In each of these pairs, the support of the first part is contained in the support of the second.
\end{enumerate}
\end{eg}

The parts of a noncrossing partition have two partially orderings: \color{blue}``vertical ordering'' defined below and ``lateral ordering'' given in Definition \ref{def: lateral ordering}.

\begin{defn}\label{def: vertical ordering}
We define the \emph{vertical ordering} of parts of a noncrossing partition by inclusion of support. Thus $X<Y$ in the vertical ordering if the support of $X$ is a subset of the support of $Y$.
\end{defn}
\color{black}

In the example we have $C<B<A$, $D<A$ and $E<A$ in the vertical ordering. The Hasse diagram of the example is thus given by:
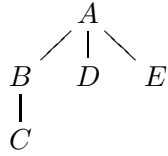
\[
\xymatrixrowsep{10pt}\xymatrixcolsep{10pt}
\xymatrix{
&A\\
B\ar@{-}[ur]& D\ar@{-}[u]& E\ar@{-}[ul]\\
C\ar@{-}[u]
	}
\]
We will say that the part $A$ of a noncrossing partition \emph{covers} part $B$ if $A$ is directly above $B$ in this Hasse diagram. In other words, $supp\,B\subset supp\,A$ and there is no part $C$ so that $supp\,B\subset supp\,C\subset supp\,A$. 

\begin{lem}
If two parts $A,B$ of a noncrossing partition are adjacent in the Hasse diagram of the partition then they are adjacent. Partially conversely, if there is another part $C$ so that $supp\,B\subset supp\,C\subset supp\,A$ then $A,B$ are not adjacent.
\end{lem}

\begin{proof}
Suppose that $A$ covers $B$. Then $supp(A\cup B)=supp\,A$. Given any other part $C$, there are three cases. (1) If $A,C$ have disjoint support then $A\cup B$ does not cross $C$. 
(2) If $supp\,A\subset supp\,C$ then $supp\,A$ does not meet $C$. Since $supp(A\cup B)=supp\,A$, the parts $A\cup B$ and $C$ do not cross. 
(3) If $supp\,C\subset supp\,A$ then $A$ is disjoint from $supp\,C$ and either (3a) $B,C$ have disjoint support or (3b) $supp\,C\subset supp\,B$ since the remaining case (3c) $supp\,B\subset supp\,C$ is excluded by assumption. In both subcases, $supp\,C$ is disjoint from $B$ as well as $A$. So, $supp\,C$ is disjoint from $A\cup B$ making them noncrossing.
Therefore, $A\cup B$ does not cross any other part of the partition.

For the partial converse, one sees immediately that $A\cup B$ and $C$ are crossing.
\end{proof}

\color{blue}
\begin{defn}\label{def: lateral ordering}
When two parts $A,B$ of a noncrossing partition have disjoint supports, the \emph{lateral ordering} is defined by saying \emph{$A$ is to the left of $B$} if every element of $A$ is less than every element of $B$ in the total ordering of $V$.
\end{defn}\color{black}

{\color{blue}In Example \ref{eg: 1.3}}, $A$ covers three parts ordered laterally as $B,D,E$. However, $B,D$ are adjacent but $D,E$ are not. This is because there is an element of $A$ between $D$ and $E$ in the lateral ordering. {\color{blue}Also, $A$ is not comparable in the lateral ordering to any other part.}

{\color{blue}
\begin{defn}\label{def: parallel parts}
We say that two parts $A,B$ of a noncrossing partition $\cS$ are \emph{parallel} if either
\begin{enumerate}
\item[(i)] $A,B$ are maximal {\color{blue}in the vertical ordering} or
\item[(ii)] $A,B$ are covered by the same part $C$ and no element of $C$ lies between $A$ and $B$ in the lateral ordering. Equivalently, $C$ is disjoint from the support of $A\cup B$.
\end{enumerate}
\end{defn}
}
Note that, in the second case, any other part $D$ above $A$ and $B$ must also be above $C$. So, $supp(A\cup B)\subseteq supp\,C$ is disjoint from $D$. 

We say that a set of pairwise parallel parts of $\cS$ is \emph{complete} if no parallel parts can be added to the set. A complete set of pairwise parallel parts will be called a \emph{parallel set}.

{\color{blue}In Example \ref{eg: 1.3}, $B$ and $D$ are parallel but $D,E$ are not parallel: All three are covered by $A$ but $D,E$ are separated by an element of $A$ whereas no element of $A$ lies between $B$ and $D$. Also, the support of $B\cup D$ is the closed interval $[2,5]$ which is disjoint from the set $\{1,6,8\}=A$, whereas the support of $D\cup E$ is $[5,7]$ which contains $6\in A$.}

\begin{prop}
Two parts of a noncrossing partition are adjacent if and only if the two parts are either parallel or one covers the other.
\end{prop}

\begin{proof} The lemma takes care of the case when $A,B$ are related in the vertical ordering. So, we may assume that $A,B$ are two parts with disjoint support. 

If $A,B$ are parallel then, for any other part $C$, there are three cases: (1) $C$ is above both $A$ and $B$. (2) $C$ is below one of then or (3) $C$ is unrelated to both by the vertical ordering. In Case (1), the support of $A\cup B$ does not contain any element of $C$ by definition of parallel. In the other two cases, the support of $C$ is disjoint from $A\cup B$. So, $A\cup B$ and $C$ do not cross. This shows that parallel parts are adjacent. 

Suppose conversely that $A,B$ are adjacent with disjoint supports. Then we claim that any other part $C$ which lies above one must lie above the other. Otherwise, $C$ and $A\cup B$ would cross. So, either $A,B$ are both maximal or they are covered by the same part $C$. In the second case we have that $C$ is disjoint from the support of $A\cup B$ since $A,B$ are adjacent. In either case, $A,B$ are parallel.
\end{proof}

We put together all adjacent pairs of parts into a single set $E(\cS)$ which we call the \emph{edge set} of the noncrossing partition $\cS$. Formally, we define a (directed) \emph{edge} to be a vector $A-B\in \ZZ\cS$ where $A,B$ are distinct parts and $\ZZ\cS$ is the free abelian group generated by $\cS$. We define $E(\cS)$ to be the set of all edges $A-B$ where either
\begin{enumerate}
\item $A,B$ are parallel parts of $\cS$ or
\item $A$ covers $B$.
\end{enumerate} 
{\color{blue}In case (1), $A-B$ and $B-A$ will both be elements of $E(\cS)$.}

\subsection{Binary forests} Although graphs have vertices and edges, when we say that ``$G$ is a graph on a set $S$'' we mean that $G$ is the set of edges of a graph and that $S$ is its set of vertices. We define an \emph{edge} (or \emph{edge vector}) to be an element of the free abelian group $\ZZ S$ of the form $w-v$. This is an edge from $v$ to $w$.

By a (directed) \emph{forest} on a finite set $S$ we mean any linearly independent set $F$ of edges $E_i=w_i-v_i\in\ZZ S$. A forest of maximal size is called a \emph{tree}. These notions are easily seen to be equivalent to the standard notions of directed graphs which are forests and spanning trees. 

Any forest $F$ on $S$ gives a partial ordering on $S$ by the condition that $v\le w$ if there is a directed path in the forest from $v$ to $w$ or, equivalently, $w-v$ is a sum of elements of $F$. A tree $T$ on $S$ is \emph{rooted} if $S$ has a unique maximal element $r$. This element is called the \emph{root} of the tree. The \emph{root vector} will be $\ast-r\in \ZZ S_+$ where $S_+=S\cup\{\ast\}$. The elements of $S$ are the \emph{vertices} of the tree but the basepoint $\ast$ is not a ``vertex''. We call $T_+=T\cup\{\ast-r\}\subset\ZZ S_+$ an \emph{augmented tree}. The tree $T\subset \ZZ S$ will be called a \emph{rooted tree}. Sometimes it will be convenient to include the root vector $\ast-r$ and sometimes not.

For an edge $w-v$ in a rooted tree, $w$ is called the \emph{parent} of $v$ and $v$ is called a \emph{child} of $w$. Note that each vertex of an augmented tree, including the root, will have exactly one parent.

We continue to assume that $V$ is a finite totally ordered set. We will put a second ``vertical'' partial order on $V$ which for notational convenience we write as follows. Let
\[
	V=\{v_1,\cdots,v_n\}.
\] 
The lateral ordering is given by the indices. For example, if we write $v_i>v_j$, $i<j$ we mean that $v_i$ is above and to the left of $v_j$.

\begin{defn}\label{def: binary tree}
We define a \emph{binary tree} on $V$ to be a rooted tree $T$ on $V$ so that {\color{blue}the partial ordering on $V$ given by $v_i<v_j$ if the unique path from $v_i$ to the root passes through $v_j$} has the following additional properties.
\begin{enumerate}
\item[(i)] {\color{blue}If the directed edge $v_j-v_i$ lies in $T$ then, $v_k<v_i$ for any $k$ between $i$ and $j$.}
\item[(ii)] Every $v_i$ has at most two children.
\item[(iii)] If $v_i$ has two children $v_j,v_k$ then $i$ lies between $j$ and $k$.
\end{enumerate}
\end{defn}

Note that, by (i), all edges are either to the left of the root or to the right of the root. Therefore, if we remove the root $r=v_k$ and edges connected to it, we will obtain two rooted binary trees, one on the set $\{v_1,\cdots,v_{k-1}\}$ and the other on the set $\{v_{k+1},\cdots,v_n\}$ (unless $k=1$ or $n$ in which case one of these two sets is empty).

{
\begin{figure}[htbp]
\begin{center}
\begin{tikzpicture}[scale=.6]
\coordinate (R8) at (8,8);
\draw[gray] (R8)--(8,0);
\coordinate (A6) at (6,6);
\coordinate (A9) at (9,6);
\coordinate (B2) at (2,4);
\coordinate (B7) at (7,4);
\coordinate (B10) at (10,4);
\coordinate (C1) at (1,2);
\coordinate (C4) at (4,2);
\coordinate (D3) at (3,0);
\coordinate (D5) at (5,0);
\foreach \x in {R8,A6,A9,B2,B7,B10,C1,C4,D3,D5}
	\draw[fill] (\x) circle[radius=3pt];
\draw[thick] (C1)--(B2)--(C4)--(D3) (C4)--(D5);
\draw[thick] (B2)--(A6)--(B7)  (A6)--(R8)--(A9)--(B10);
\draw (R8) node[above]{$v_8$};
\draw (5.7,6) node[above]{$v_6$};
\draw (A9) node[right]{$v_9$};
\draw (B2) node[left]{$v_2$};
\draw (B7) node[left]{$v_7$};
\draw (B10) node[left]{$v_{10}$};
\draw (C1) node[left]{$v_1$};
\draw (C4) node[right]{$v_4$};
\draw (D3) node[left]{$v_3$};
\draw (D5) node[right]{$v_5$};
\draw (0,7) node{$T:$};
\end{tikzpicture}
\color{blue}
\caption{Binary tree on 10 vertices. The root is $v_8$. Removing the root gives two subtrees. }
\label{Fig: binary tree example}
\end{center}
\end{figure}
}

{\color{blue}
\begin{eg}\label{eg: a binary tree}
Figure \ref{Fig: binary tree example} gives one example of a binary tree on the set $V=\{v_1,\cdots,v_{10}\}$. The root is $v_8$ and, if we remove the root, we obtain two binary trees, one on $v_1,v_2,\cdots,v_7$ and the other on $v_9,v_{10}$. Condition (i) says that, since $E=v_6-v_2\in T$, $v_3,v_4,v_5$ are $<v_2$ (i.e., they are descendants of $v_2$, the lower vertex of $E$). Condition (iii) says that if a vertex has two children, one is on the left and one is on the right. For example, $v_2$ has two children: a left child $v_1$ and a right child $v_4$.
\end{eg}
}

Given a noncrossing partition $\cS$ on $V$, every parallel subset $P$ of $\cS$ is totally ordered {\color{blue}by the lateral ordering. (Recall that a parallel subset of $\cS$ is a maximal set of pairwise parallel parts in $\cS$).} So, it makes sense to talk about a binary tree on $P$.

\begin{defn}
A \emph{binary forest} on a noncrossing partition $\cS$ on $V$ is define to be the union of binary trees, one on each parallel subset of $\cS$.
\end{defn}

For example, if the parts of $\cS$ have disjoint support then a binary forest on $\cS$ is the same as a binary tree on $\cS$. In other cases, we need to add edges to the forest to get a tree.

\begin{prop}
Let $F$ be a binary forest on a noncrossing partition $\cS$. Then $F\subseteq E(\cS)$ and there exists a unique rooted tree $T$ on $\cS$ so that $F\subseteq T\subseteq E(\cS)$.
\end{prop}

\begin{proof} Since $E(\cS)$ contains all edges whose endpoints are parallel, $F$ is contained in $E(\cS)$.

The root of the parallel set of all maximal elements of $\cS$ must be the root of $T$ since no edges in $E(\cS)$ start at a maximal part by definition. For every other  parallel set of $\cS$, we need to add an edge starting at its root. This edge must point to the part which covers the parallel set. After adding these required edges, we have the unique rooted tree $T\subset E(\cS)$ containing $F$.
\end{proof}

\begin{rem}\label{rem: rooted tree generated by F}
We call $T$ the \emph{rooted tree generated by $F$}. Then $T$ is a basis for the kernel of the augmentation map $\vare:\ZZ\cS\to \ZZ$ given by $\vare(A)=1$ for every $A\in\cS$.
\end{rem}

\subsection{Binary forests as morphisms}

Suppose that $\cS,\cR$ are noncrossing partitions of $V$ and $\cS$ is a refinement of $\cR$. Then we have an epimorphism of sets $\pi:\cS\onto \cR$ sending each part of $\cS$ to the unique part of $\cR$ which contains it. 

\color{blue} 
\begin{defn}\label{def: cluster morphism T:R to S}We define a \emph{cluster morphism} $[T]:\cR\to \cS$ to be the union $T=\coprod T_W$ of a set of rooted trees $T_W$ obtained in the following way. Each $W\in \cR$ is totally ordered being a subset of $V$. And $\pi^{-1}(W)\subseteq\cS$ is a noncrossing partition of $W$. We choose a binary forest $F_W$ on each $\pi^{-1}(W)$ and let $T_{W}$ be the rooted tree that it generates. The disjoint union $T=\coprod T_W$ of a set of rooted trees $T_w$ chosen in this way is one cluster morphism $[T]:\cR\to \cS$. The nomenclature is explained in Remark \ref{rem: cluster morphism nomenclature}.
\end{defn}
\color{black}

{\color{blue} Since $T_W\subset E(\pi^{-1}(W))\subset E(\cS)$, $T$ will be a subset of $E(\cS)$}. We need the following relative version of the set $E(\cS)$.
\[
	E(\cS,\cR):=\coprod_{W\in\cR}E(\pi^{-1}(W))
\]
Then $E(\cS,\cR)$ is a subset of the kernel of the linear epimorphism $\pi_\ast:\ZZ\cS\onto \ZZ\cR$ induced by $\pi$. Applying the discussion of the previous subsection to each noncrossing partition $\pi^{-1}(W)$ separately, we get the following.

\begin{prop}
The edges (elements) of any cluster morphism $[T]:\cR\to \cS$ lie in $E(\cS,\cR)$ and form a basis for the kernel of $\pi_\ast$. In particular, $|T|=rk\,\cR-rk\,\cS$.\qed
\end{prop}

\begin{lem}
Let $\cR,\cS$ be as above. Then $E(\cS,\cR)=E(\cS)\cap \ker \pi_\ast$. In particular, $E(\cS,\cR)\subset E(\cS)$.
\end{lem}

\begin{proof}
Suppose that $A,B$ are parts of $\cS$ which lie in one part $W\in\cR$. Then, it follows from Proposition \ref{characterization of refinement} that $A,B$ are adjacent as parts of the noncrossing partition $\pi^{-1}(W)$ of $W$ if and only if they are adjacent as parts of $\cS$. Therefore, $A-B$ lies in $E(\cS,\cR)$ if and only if it lies in $E(\cS)$. The lemma follows.
\end{proof}

\begin{prop}
Let $\cR,\cS$ be as above and suppose $\cR$ is a refinement of $\cQ$. Then $E(\cS,\cR)=E(\cS,\cQ)\cap\ker\pi_\ast$. In particular, $E(\cS,\cR)$ is a subset of $E(\cS,\cQ)$.\qed
\end{prop}

A small category is defined to be a set of objects, a set of morphisms and a law of composition. We come to the last and most difficult part of the definition.

\subsection{Composition of cluster morphisms}

We define the \emph{noncrossing partition category} $\cN\cP(n)$ to be the category whose \emph{objects} are the noncrossing partitions of $\{1,\cdots,n\}$, whose \emph{morphisms} $[T]:\cR\to \cS$ are the cluster morphisms defined above with composition defined below assuming two theorems \ref{thm1: morphisms are maximal compatible sets of edge vectors} and \ref{thm2: bijection ET to E} which we will prove later.

{\color{blue}
We will use the following example to illustrate composition of cluster morphisms.

\begin{eg}\label{eg: composition of morpisms}
Let $\cP$ be the partition of $V=\{1,2,\cdots,8\}$ with one part $V$. Let $\cQ$ be partition of $V$ with two parts $P=\{2,3,4,5,6\}$ and $U=\{1,7,8\}$. Since $U$ covers $P$, there is only one morphism $[R]:\cP\to \cQ$ given by 
\[
R=\{U-P\}.\]
{
\begin{figure}[htbp]
\begin{center}
{
\begin{tikzpicture}

\begin{scope}
	\foreach \x in {0.55,2.55,4.55,5.55}
	{\draw[thick, rounded corners=.35cm] (\x,-.35) rectangle +(.9,.75);
	} 
	\draw (-1.5,.5) node{$\cS=$}
	(0,0) node{1}
	(1,0) node{$B$:2}
	(2,0) node{3}
	(2,.8) node{$C$}
	(3,0) node{$D$:4}
	(4,0) node{5}
	(5,0) node{$E$:6}
	(6,0) node{$F$:7}
	(7,0) node{8};
	\draw[thick,rounded corners=.35cm]  (-.4,2)--(-.4,-.4)--(.4,-.4)--(.4,1.5)--(4.6,1.5)--(5.4,1.5)--(6.6,1.5)--(6.6,-.4)--(7.4,-.4)--(7.4,2)--cycle;
	\draw (0,1.5) node{$A$};
\begin{scope}[xshift=1cm]
	\draw[thick,rounded corners=.35cm]  (.6,1.2)--(.6,-.4)--(1.4,-.4)--(1.4,.7)--(2.6,.7)--(2.6,-.4)--(3.4,-.4)--(3.4,1.2)--cycle;
	\end{scope} 
\end{scope} 
\begin{scope}[yshift=-2cm]
	\foreach \x in {0.55,5.55}
	{\draw[thick, rounded corners=.35cm] (\x,-.35) rectangle +(.9,.75);
	} 
	{\draw[thick, rounded corners=.35cm] (1.55,-.35) rectangle +(3.9,.75);
	} 
	\draw (-1.5,.5) node{$\cR=$}
	(0,0) node{1}
	(1,0) node{$B$:2}
	(2,0) node{3}
	(2.7,0) node{4}
	(3.5,0) node{$W$}
	(4.3,0) node{5}
	(5,0) node{6}
	(6,0) node{$F$:7}
	(7,0) node{8};
	\draw[thick,rounded corners=.35cm]  (-.4,1)--(-.4,-.4)--(.4,-.4)--(.4,.5)--(6.6,.5)--(6.6,-.4)--(7.4,-.4)--(7.4,1)--cycle;
	\draw (0,0.6) node{$A$};

\end{scope} 
\end{tikzpicture}
}
\color{blue}
\caption{Two refinements of $\cQ=\{P,U\}$ given by $\cR=\{B,W,F,A\}$ where $P=B\cup W$ and $U=A\cup F$ and $\cS$ which is a further refinement of $\cR$ given by subdividing $W$ into $W=C\cup D\cup E$.}
\label{Fig: First example of partitions R and S}
\end{center}
\end{figure}
}

Consider the two refinements $\cR$, $\cS$ of $\cQ$ given in Figure \ref{Fig: First example of partitions R and S}. Since $\cS$ is a refinement of $\cR$ we have $\pi:\cS\onto \cR$. $\pi^{-1}(W)$ has two parallel sets $\{C,E\}$ and $\{D\}$. Let $[T]:\cR\to \cS$ be the cluster morphism given by the binary tree making $E$ the right child of $C$. Since $C$ covers $D$,
\[
\xymatrixrowsep{1pt}\xymatrixcolsep{10pt}
\xymatrix{
& &
	C\ar@{-}[dd]\ar@{-}[ddr]\\
T= &
	\{C-D,C-E\}:\\	
& 
	&
	D & E.
	}
\]
Since $\cR$ is a refinement of $\cQ$ we have $\rho:\cR\onto \cQ$. Since $\rho^{-1}(P)=\{B,W\}$ is a parallel set and $A$ covers $F$ in $\rho^{-1}(U)$, there are two morphisms $\cQ\to \cR$ and one of them is $[S]$ given by taking $W$ over $B$. Since $B$ is to the left of $W$, it becomes the left child of $W$:
\[
\xymatrixrowsep{1pt}\xymatrixcolsep{10pt}
\xymatrix{
& &
	&W\ar@{-}[ddl] & A\ar@{-}[dd]\\
S= &
	\{W-B,A-F\}:\\	
& 
	&
	B && F.
	}
\]
We will compute the composition of these three morphisms
\[
	\cP\xrightarrow{[R]}\cQ\xrightarrow{[S]}\cR\xrightarrow{[T]}  \cS.
\]
\end{eg}}

\begin{defn}
Let $\cR,\cS$ be objects of $\cN\cP(n)$ so that $\cS$ is a refinement of $\cR$. We define two edges $E,E'\in E(\cS,\cR)$ to be \emph{compatible} if there is a cluster morphism from $\cR$ to $\cS$ which contains both of them.
\end{defn}

{\color{blue} In Example \ref{eg: composition of morpisms}, $E(\cS,\cR)=\{C-D,C-E,E-C\}$. There are two compatible pairs: $\{C-D,C-E\}$ and $\{C-D,E-C\}$. Indeed, the first pair gives the cluster morphism $[T]$.}
\begin{thm}\label{thm1: morphisms are maximal compatible sets of edge vectors}
A subset $T$ of $E(\cS,\cR)$ gives a cluster morphism $[T]:\cR\to\cS$ if and only if the elements of $T$ are pairwise compatible and $T$ is maximal with this property.
\end{thm}

\begin{defn}
Let $\cQ,\cR,\cS$ be objects of $\cN\cP(n)$ so that $\cS$ is a refinement of $\cR$ and $\cR$ is a refinement of $\cQ$. Let $[T]:\cR\to\cS$ be a cluster morphism. Recall that $T\subset E(\cS,\cR)\subseteq E(\cS,\cQ)$. We define $E_T(\cS,\cQ)$ to be the set of all element of $E(\cS,\cQ)$ which are compatible with the elements of $T$ but not contained in $T$. 
\end{defn}
{\color{blue}
In Example \ref{eg: composition of morpisms}, $E(\cS,\cQ)=\{A-F,C-D,B-C,C-B,C-E,E-C,B-E,E-B\}$. However, $T$ is not compatible with $E-C, B-E,E-B$. So, 
\[
E_T(\cS,\cQ)=\{A-F,B-C,C-B\}.\]}
\begin{thm}\label{thm2: bijection ET to E}
The linear map $\pi_\ast:\ZZ\cS\to\ZZ\cR$ induces a bijection $E_T(\cS,\cQ)\cong E(\cR,\cQ)$. Furthermore, $E,E'$ are compatible edges in $E_T(\cS,\cQ)$ if and only if the corresponding elements of $E(\cR,\cQ)$ are compatible.
\end{thm}

Let 
\[
\mu_T: E(\cR,\cQ)\to E_T(\cS,\cQ)
\]
be the inverse of the bijection given by the theorem. {\color{blue}
In our example, $\pi:\cS\to\cR$ sends $C$ to $W$ and therefore sends the element $A-F,B-C,C-B$ of $E_T(\cS,\cQ)$ to $A-F,B-W,W-B\in E(\cR,\cQ)$ respectively. $\mu_T$ is the inverse bijection.}

\begin{defn}
The composition of cluster morphisms $[S]:\cQ\to \cR$ and $[T]:\cR\to\cS$ is given by:
\[
	[T]\circ[S]=[T\cup \mu_TS]:\cQ\to\cS.
\]
We verify that this is a cluster morphism. By Theorem \ref {thm1: morphisms are maximal compatible sets of edge vectors}, $S$ is a maximal compatible subset of $E(\cR,\cQ)$. By Theorem \ref {thm2: bijection ET to E}, $\mu_T(S)$ is a maximal compatible subset of $E_T(\cS,\cQ)$. This is equivalent to the statement that $T\cup \mu_T(S)$ is a maximal compatible subset of $E(\cS,\cQ)$ which implies that $[T\cup \mu_T(S)]$ is a cluster morphism $\cQ\to\cS$.
\end{defn}

{\color{blue}
In our example, $\mu_T: E(\cR,\cQ)\to E_T(\cS,\cQ)$, computed above, sends $S=\{W-B,A-F\}$ to $\{C-B,A-F\}$. So:
\[
\xymatrixrowsep{1pt}\xymatrixcolsep{10pt}
\xymatrix{
& &
	&C\ar@{-}[ddl]\ar@{-}[dd]\ar@{-}[ddr] && A\ar@{-}[dd]\\
T\cup \mu_T(S)= &
	\{C-E,C-D,C-B,A-F\}:\\	
& 
	&
	B & D&E& F.
	}
\]}

{
\begin{figure}[htbp]
\begin{center}
\begin{tikzpicture}[scale=.7]
\coordinate (v2) at (1.5,2);
\coordinate (v3) at (3,3);
\coordinate (v4) at (4.5,0);
\coordinate (v5) at (6,1);
\coordinate (v6) at (7.5,2);
\coordinate (v7) at (9,4);
\coordinate (v8) at (10.5,5);
\coordinate (b23) at (2,2.5);
\coordinate (b56) at (5,.5);
\coordinate (b47) at (5.25,2.5);
\coordinate (b78) at (10.1,4.6);
\draw (b23) node[above]{$\beta_{23}$};
\draw (b56) node[above]{$\beta_{45}$};
\draw (b47) node[below]{$-\beta_{36}$};
\draw (b78) node[below]{$\beta_{78}$};
\foreach \x in {v2,v3,v4,v5,v6,v7,v8}
	\draw[fill] (\x) circle[radius=3pt];
\draw[thick] (v2)--(v3)--(v6);
\draw[thick] (v4)--(v5) (v7)--(v8);
\draw[dotted] (v5)--(v6) (v3)--(v7);
\draw (v2)node[left]{2};
\draw (v4)node[left]{4};
\draw (v5)node[right]{5};
\draw (v3)node[above]{3};
\draw (v6)node[right]{6};
\draw (v7)node[below]{7};
\draw (v8)node[right]{8};
\end{tikzpicture}
\color{blue}
\caption{$T\cup \mu_T(S)= C-E,C-D,C-B,A-F$ in binary tree notation (Definition \ref{def: binary tree notation}). We can immediately see that these edges are compatible since they can be completed to a binary tree with edges indicated with dotted lines. (See Definition \ref{def: compatibility for G(V)} and Theorem \ref{thm: maximal compatible sets are rooted binary trees}.)}
\label{Fig: binary tree notation}
\end{center}
\end{figure}
}

{\color{blue}
\begin{defn}\label{def: binary tree notation}
Since this graph does not look like a binary tree, this may be a good moment to introduce \emph{binary tree notation} for the same graph. (See Figure \ref{Fig: binary tree notation}.) The edges $C-B$, $C-E$ are replaced with $\beta_{23}:=v_3-v_2$ and $v_3-v_6=-\beta_{36}$ using the convention that parallel parts, say $C,E$, are replaced by their left-most vertices $v_3,v_6$ and $\beta_{ij}=v_j-v_i$ (so $v_i-v_j=-\beta_{ij}$). The edges $C-D$, $A-F$ are replaced with $\beta_{45}=v_5-v_4$ and $\beta_{78}=v_8-v_7$ using the convention that, when $X$ covers $Y$, we write the edge $X-Y$ as $v_k-v_j$ where $v_j$ is the left-most vertex of $Y$ and $v_k$ is the left-most vertex of $X$ which is to the right of $Y$. Thus, $C-D$ becomes $v_5-v_4$ since $v_5$ is the left-most vertex of $C=\{3,5\}$ to the right of $D=4$. As an application of Theorem \ref{thm: maximal compatible sets are rooted binary trees} below, the resulting graph can be completed to a binary graph (by adding the edges indicated with dotted lines) since its edges are pairwise compatible according to Definition \ref{def: compatibility for G(V)}.
\end{defn}
}

Next, we will verify that composition of morphisms is associative. Let
\[
	\cP\xrarrow{[R]} \cQ\xrarrow{[S]} \cR\xrarrow{[T]} \cS
\]
be three composable cluster morphisms. Then\color{blue}
\[
	[T]\circ([S]\circ[R])=[T]\circ [R\cup \mu_RS]=[R\cup \mu_RS\cup \mu_{R\cup \mu_RS}T]
\]
\[
	([T]\circ[S])\circ[R]=[S\cup \mu_ST]\circ [R]=[R\cup \mu_RS\cup \mu_R\mu_ST]
\]
So, it suffices to show that $\mu_R\circ\mu_S=\mu_{R\cup\mu_RS}$. But, the first map is the composition of two bijections
\[
	E(\cQ,\cP)\xrarrow{\mu_S} E_S(\cR,\cP)\xrarrow{\mu_R} E_{R\cup \mu_RS}(\cS,\cP)
\] whose inverses are given by the projection maps $\ZZ\cS\onto \ZZ\cR\onto \ZZ\cQ$. And, the second map is the bijection $E(\cQ,\cP)\cong E_{R\cup \mu_RS}(\cS,\cP)$ whose inverse is given by the projection map $\ZZ\cS\onto \ZZ\cQ$. So, these two maps agree.

In our example, we computed $X=T\cup \mu_TS$ above. 
\[
	\mu_X:E(\cQ,\cP)=\{U-P\}=R\to E_X(\cS,\cP)=\{F-C\}.
\]
So, $\mu_X(R)=F-C$ and $([T]\circ [S])\circ [R]=[X]\circ[R]=[X\cup \mu_X(R)]$ where
\[
	X\cup \mu_X(R)=\{C-E,C-D,C-B,A-F,F-C\}.
\]

On the other side we have $[S]\circ[R]=[S\cup \mu_S(R)]$. But 
\[
\mu_S:E(\cQ,\cP)=\{U-P\}=R\to E_S(\cR,\cP)=\{F-W\}.
\]
So, $[S]\circ[R]=[Y]$ where
\[
	Y=S\cup \mu_S(R)=\{W-B,A-F,F-W\}.
\]
We have $[T]\circ([S]\circ[R])=[T]\circ [Y]=[T\cup \mu_T(Y)]$ where
\[
	\mu_T:E(\cR,\cP)\to E_T(\cS,\cP)
\]
sends $W-B,A-F,F-W$ to $C-B,A-F,F-C$. Therefore,
\[
	T\cup \mu_T(Y)=\{C-D,C-E,C-B,A-F,F-C\}
\]
which is equal to $X\cup \mu_X(R)$. Therefore, $([T]\circ[S])\circ[R]=[T]\circ([S]\circ [R])$ in this case.

\color{black}Therefore, the definition of the category of noncrossing partitions will be complete when we prove Theorems \ref {thm1: morphisms are maximal compatible sets of edge vectors} and \ref {thm2: bijection ET to E}.

\section{Compatibility of edge sets}

The purpose of this section is to prove Theorem \ref{thm1: morphisms are maximal compatible sets of edge vectors} and Theorem \ref {thm2: bijection ET to E}, the two theorems needed to define composition of cluster morphisms. The same idea is also needed in the proof that the classifying space of the category of noncrossing partitions is locally $CAT(0)$. We will first do the basic case of Theorem \ref{thm1: morphisms are maximal compatible sets of edge vectors} to show that binary trees are maximal compatible sets of edges.

\subsection{Binary trees as maximal compatible sets}

By definition a binary tree $T$ has a unique root, say $v_k$. We will augment this tree with the root vector $\ast-v_k$. If $V=\{v_1,\cdots,v_n\}$ is a finite totally set (the set of vertices), let $V_+=V\cup\{\ast\}$. Define $E(V)$ to be the subset of the free abelian group $\ZZ V\cong \ZZ^n$ consisting of the $n(n-1)$ vectors $v_j-v_i$ for $i\neq j$ which we call \emph{edge vectors}. Let $G(V)\subset \ZZ V_+\cong \ZZ^{n+1}$ be the union of $E(V)$ and the $n$ vectors $\ast-v_k$ which we call \emph{root vectors}. So, $|G(V)|=n^2$. We define an \emph{augmented binary tree} $T_+$ on $V$ to be a subset of $G(V)$ with $n$ elements: 
\begin{enumerate}
\item the edge vectors $v_j-v_i$ of the $n-1$ directed edges $v_i\to v_j$ of the tree and 
\item the (unique) root vector $\ast-v_k$ of the root $v_k$ of the tree. 
\end{enumerate}
As before the \emph{support} of $E=v_i-v_j$ is the closed interval in $\RR$ with endpoints $i,j$. The support of a root vector $\ast-v_k$ is $[k,\infty)$ and the length of the root vector is infinite.

The definition we want is: ``Two elements of $G(V)$ are {compatible} if there exists a binary tree which contains both of them.'' For example, $v_1-v_2$ is compatible with $\ast-v_1$ but not with $\ast-v_2$ since the inclusion of the edge $v_1-v_2$ implies $v_1>v_2$ in the partial ordering given by the tree. {\color{blue}Indeed, $v_1-v_2$ being a directed edge in the tree means $v_1$ is closer to the root than $v_2$.} So, $v_1$ can be a root, but not $v_2$. The theorem we want is: ``Binary trees are the same as maximal compatible subsets of $G(V)$.'' To prove this we need a more precise statement.

\begin{defn}\label{def: compatibility for G(V)}
\emph{Compatibility} of pairs of elements of $G(V)$ is defined by the following conditions.
\begin{enumerate}
\item A root vector $\ast-v_k$ is not compatible with any other root vector.
\item A root vector $\ast-v_k$ and an edge $E=v_j-v_i$ are compatible if and only if either $k=j$ or $k$ lies outside the support of $E$. {\color{blue}I.e., either $k<i,j$ or $i,j<k$. Equivalently, $v_j-v_i$ and $\ast-v_k$ are noncrossing (Condition (3)(a) below).}
\item Given two edges $E_1$ and $E_2$ there are several possibilities:
\begin{enumerate}
	\item (noncrossing condition) If $E_1,E_2$ have distinct endpoint, they are compatible if and only if either their supports are disjoint or the support of one is contained in the interior of the support of the other. 
	\item If the intersection of the supports of $E_1,E_2$ is one point $v_j$ then they are compatible if and only if they do not both point away from $v_j$.
	\item If two edges share one endpoint $v_k$ and the support of one is contained in the support of the other then they are compatible if and only if they have different lengths and the longer edge points away from $v_k$ and the shorter edge points towards $v_j$.
\end{enumerate} 
\end{enumerate}
\end{defn}

{
\begin{figure}[htbp]
\begin{center}
\begin{tikzpicture}
\draw[thick] (-7.5,3) rectangle (7.5,3.5);
\draw (-2,3.25) node{compatible};
\draw (4.75,3.25) node{not compatible};
\draw[thick] (-7.5,0) rectangle (7.5,3);
\draw[thick] (-6,0)--(-6,3) (2,0)--(2,3.5);
\draw (-6.75,1.5) node{(3)(a)};
\begin{scope}[xshift=-5.5cm,yshift=1.5cm]
\draw[very thick,->] (0,0)--(1,1);
\draw[very thick,->] (2.5,0)--(1.5,1);
\draw[very thick,gray] (0,-.5)--(1,-.5) (2.5,-.5)--(1.5,-.5);
\draw[gray] (1.25,-1) node{\tiny disjoint support OK};
\end{scope}
\begin{scope}[xshift=-1.9cm,yshift=1.5cm]
\draw[very thick,->] (0,0.2)--(3,1);
\draw[very thick,->] (1,0)--(2,.25);
\draw[very thick,gray] (0,-.5)--(3,-.5) (1,-.65)--(2,-.65);
\draw[gray] (1.5,-1) node{\tiny nested support OK};
\end{scope}
\begin{scope}[xshift=3.2cm,yshift=1.5cm]
\draw[very thick,->] (0,0)--(2,1);
\draw[very thick,->] (1,0)--(3,.25);
\draw[very thick,gray] (0,-.5)--(2,-.5) (1,-.65)--(3,-.65);
\draw[gray] (1.5,-1) node{\tiny overlapping support not OK};
\end{scope}
\draw[thick] (-7.5,0) rectangle (7.5,-3);
\draw[thick] (-6,0)--(-6,-3) (2,0)--(2,-3);
\draw (-6.75,-1.5) node{(3)(b)};
\begin{scope}[xshift=3.2cm,yshift=-1.5cm]
\draw[very thick,->] (1.5,0.1)--(2.5,1);
\draw[very thick,->] (1.5,0.1)--(.5,1);
\draw[fill] (1.5,0.1) circle[radius=2pt];
\draw (1.5,-.2) node{$v_j$};
\draw[gray] (1.5,-.7) node{\tiny a vertex cannot have};
\draw[gray] (1.5,-1.1) node{\tiny two parents};
\end{scope}
\begin{scope}[xshift=-1.5cm,yshift=-1.5cm]
\draw[very thick,->] (0,0)--(1,.6);
\draw[very thick,->] (1.1,.6)--(2,1.2);
\draw[fill] (1.1,0.6) circle[radius=2pt];
\draw (1,.3) node[right]{$v_j$};
\draw[gray] (-.8,-1) node{\tiny these are OK};
\end{scope}
\begin{scope}[xshift=-5.5cm,yshift=-1.7cm]
\draw[very thick,->] (0.3,-.1)--(1.17,.8);
\draw[very thick,->] (2.3,-.1)--(1.34,.8);
\draw[fill] (1.25,0.85) circle[radius=2pt];
\draw (1.25,0.85) node[above]{$v_j$};
\end{scope}
\draw[thick] (-7.5,-6) rectangle (7.5,-3);
\draw[thick] (-6,-6)--(-6,-3) (-2,-6)--(-2,-3);
\draw (-6.75,-4.5) node{(3)(c)};
\begin{scope}[xshift=-5cm,yshift=-4.2cm]
\draw[very thick,->] (0,0)--(2,.8);
\draw[very thick,<-] (0.1,-.1)--(1,-.5);
\draw[fill] (0,0) circle[radius=2pt];
\draw (0,0) node[left]{$v_k$};
\draw[gray] (1,-1) node{\tiny OK: child is closer to $v_k$};
\draw[gray] (1,-1.3) node{\tiny than parent};
\end{scope}
\begin{scope}[xshift=-1cm,yshift=-4.2cm]
\draw[very thick,->] (0,0)--(1,.8);
\draw[very thick,<-] (0.1,-.1)--(2,-.5);
\draw[fill] (0,0) circle[radius=2pt];
\draw (0,0) node[left]{$v_k$};
\end{scope}
\begin{scope}[xshift=2cm,yshift=-4.4cm]
\draw[very thick,->] (0,0)--(2,.6);
\draw[very thick,->] (0,0)--(1,.7);
\draw[fill] (0,0) circle[radius=2pt];
\draw (0,0) node[left]{$v_k$};
\draw[gray] (1,-1) node{\tiny these are not OK};
\end{scope}
\begin{scope}[xshift=5cm,yshift=-4.2cm]
\draw[very thick,<-] (0.1,0)--(2,-.5);
\draw[very thick,<-] (0.1,-.1)--(1,-.8);
\draw[fill] (0,0) circle[radius=2pt];
\draw (0,0) node[left]{$v_k$};
\end{scope}
\end{tikzpicture}
\color{blue}
\caption{Compatibility chart showing pairwise compatibility of edges in augmented binary trees}
\label{Figure99}
\end{center}
\end{figure}
}

Conditions (3b), (3c) can be combined into one condition as follows.\vs2
{\color{blue}
\begin{minipage}{0.6\textwidth}
\begin{enumerate}
\item[(3bc)] Every vertex $v_j$ has at most one left child ($v_i$), at most one right child ($v_k$) and at most one parent ($v_\ell$). Furthermore a child on the same side as a parent must be strictly closer to $v_j$ than the parent.
\end{enumerate}
\end{minipage}
\hfill 
\begin{minipage}{0.35\textwidth}
\begin{center}
\begin{tikzpicture}
\draw[very thick,->] (0,0.1)--(2,.6);
\draw[very thick,<-] (0.1,0)--(1,-.6);
\draw[very thick,<-] (-.1,0)--(-1,-.6);
\draw[fill] (0,.1) circle[radius=2pt];
\draw[fill] (-1,-.6) circle[radius=2pt];
\draw[fill] (1,-.6) circle[radius=2pt];
\draw[fill] (2.1,.6) circle[radius=2pt];
\draw (0,.1) node[above]{$v_j$};
\draw (-1,-.6) node[left]{$v_i$};
\draw (1,-.6) node[right]{$v_k$};
\draw (2.1,.6) node[right]{$v_\ell$};
\end{tikzpicture}
\end{center}
\end{minipage}%
}\vs2

It follows from the definition of a binary tree that the edges and root vector of a binary tree on $V$ form a pairwise compatible subset of $G(V)$. We will prove the converse:

\begin{thm}\label{thm: maximal compatible sets are rooted binary trees}
Any maximal pairwise compatible subset of $G(V)$ has $n$ elements and consists of the edges and root vector of a binary tree on $V$.
\end{thm}

Some immediate consequences of this theorem are the following.

\begin{cor}
\begin{enumerate}
\item Every compatible pair of elements of $G(V)$ lie in an augmented binary tree on $V$.
\item The root vector and edges of any binary tree on $V$ form a maximal compatible subset of $G(V)$.
\item Any maximal compatible subset of $E(V)$ has $n-1$ elements and consists of the edges of a binary tree on $V$.
\end{enumerate}
\end{cor}

\begin{proof}
(1) Any compatible set is contained in a maximal compatible set which forms an augmented binary tree by the theorem.

(2) The parts of an augmented binary tree are pairwise compatible. So, they form a subset of a maximal pairwise compatible set which, by the theorem, must be the set we started with.

(3) Any maximal compatible subset $T$ of $E(V)$ is contained in a maximal compatible subset of $G(V)$ which contains one root vector. Removing the root vector must give $T$.
\end{proof}

We prove Theorem \ref{thm: maximal compatible sets are rooted binary trees} by induction on the size of $V$ using the root.
{\color{blue}Given any set of directed edges $S\subset E(V)$, let $\Gamma(S)$ be the directed graph with vertex set $V$ and edge set $S$.}

\begin{lem}\label{lem: Gamma(S)}
\color{blue}(a) Given any pairwise compatible subset $S$ of $E(V)$, every component of $\Gamma(S)$ has a unique root (a vertex without a parent).

(b) Every component of $\Gamma(S)$ is a tree.

(c) Let $v_k$ be the root of the component of $\Gamma(S)$ which contains the left-most vertex $v_1$. Then the root vector $\ast-v_k$ is compatible with $S$. Equivalently, the vertical line though $v_k$ does not cross any edge in $S$.
\end{lem}

{\color{blue} 
\begin{proof} 
(a) Let $\Gamma_0$ be any component of $\Gamma(S)$. The uniqueness of a root, if it exists, follows from Condition (3)(b): If there were two roots in the same component, there would be path joining them. Since the edges of that path begin by pointing left and end by pointing right, there must be a vertex at which this switches. This vertex would then have two parents which contradicts (3)(b). Thus, only existence of the root is in question. 

Let $v_i$ be the left-most vertex in $\Gamma_0$. So, $v_i$ has no left child and no left parent. If $v_i$ has no right parent, it is a root. So, we may assume $v_i$ has a right parent, say $v_j$. The vertex $v_j$ cannot have a left parent since any left parent $v_p$ would have $i<p<j$ (since $i$ is minimal) and this contradicts (3)(c). If $v_j$ has no right parent, we are done. So, suppose $v_j$ has a right parent $v_p$. Then $v_p$ cannot have a left parent $v_q$ since $i<q<p$ forces $j<q$ since otherwise $[i,j]$ and $[q,p]$ would cross, contradicting (3)(a). But then $j<q<p$ gives a contradiction to (3)(c). Continuing in this way we reach vertex $v_k$ with no right parent and no left parent. So, $v_k$ is a root of $\Gamma_0$.

(b) Every component of $\Gamma(S)$ must be a tree. Otherwise, it would have a cycle. By part (a) that cycle must have a root. But then, it must also have a minimal element which has two parents, which is not allowed. So, $\Gamma(S)$ has no cycles. So, it is a tree.

(c) Let $v_k$ be the root of the component of $\Gamma(S)$ containing the first vertex $v_1$. Referring to the proof of (a) above, we know that $v_k$ is obtained from $v_1$ by taking a sequence of right parents $v_1=v_{j_0},v_{j_1},v_{j_2},\cdots,v_{j_m}=v_k$ where each $v_{j_p}$ is the right parent of $v_{j_{p-1}}$. If any other edge $E$ in $S$ crosses the vertical line at $v_k$, then the endpoints of $E$ must be $v_i,v_j$ where $i<k<j$. But $v_i$ cannot be in the sequence of right parents since that would contradict (3)(c). Since $1<i<k$, $v_i$ must fall between two consecutive right parents $v_{j_p}$ and $v_{j_{p+1}}$. Then $[i,j]$ would cross the support $[j_p,j_{p+1}]$ of this edge contradicting (3)(a). So, the vertical line through $v_k$ does not cross any of the edges in $S$ making the root vector $\ast-v_k$ compatible with $S$.
\end{proof}}

{
\begin{figure}[htbp]
\begin{center}
\begin{tikzpicture}
\draw[gray] (1,3)--(1,0);
\draw[gray] (3,2)--(3,0);
\draw[gray, thick,->] (-1,2)--(.9,3);
\draw[fill] (-1,2) circle[radius=2pt];
\draw (-1,2) node[left]{$v_1$};
\draw[fill] (0,1) circle[radius=2pt];
\draw[very thick,->] (0,1)--(-.9,1.9);
\draw (-.5,0) node{$\Gamma_L$};
\draw[gray, thick,->] (3,2)--(1.1,3);
\draw[fill,gray] (1,3) circle[radius=2pt];
\draw (1,3) node[above]{$v_k$};
\draw (4.5,3) node{$\Gamma_R$};
\draw[very thick,->] (2,1)--(2.9,2);
\draw[fill] (3,2) circle[radius=2pt];
\draw (3,2) node[above]{$v_j$};
\draw[fill] (2,1) circle[radius=2pt];
\draw (2,1) node[below]{$v_{k+1}$};
\draw[fill] (5,1) circle[radius=2pt];
\draw (5,.8) node[below]{$v_i$};
\draw[fill] (7,1) circle[radius=2pt];
\draw (7,1) node[below]{$v_{n}$};
\draw[very thick,->] (7,1)--(3.1,2);
\draw[very thick,->] (4,0)--(4.9,1);
\draw[very thick,->] (6,0)--(5.1,1);
\end{tikzpicture}
\color{blue}
\caption{If $\Gamma(S)$ is not connected, we remove the root $v_k$ to get two subgraphs $\Gamma_L$ and $\Gamma_R$ one of which will not be connected, contradicting the induction hypothesis.}
\label{Fig: If Gamma(S) is not connected}
\end{center}
\end{figure}
}

{\color{blue}
\begin{eg}\label{eg: when Gamma(S) is not connected} 
When $S\subset E(V)$ is a maximal pairwise compatible set, we want to show that $\Gamma(S)$ is connected. Figure \ref{Fig: If Gamma(S) is not connected} shows an example of what happens when $\Gamma(S)$ is not connected. 
By Lemma \ref{lem: Gamma(S)}(c) there is at least one root $v_k$ of some component of $\Gamma(S)$ having the property that $\ast-v_k$ is compatible with $S$. If we remove the vertex $v_k$ and the edges ending in $v_k$ (shaded gray in Figure \ref{Fig: If Gamma(S) is not connected}) we will get a left graph $\Gamma_L$ with vertex set $L=\{v_1,v_2,\cdots,v_{k-1}\}$ and a right graph $\Gamma_R$ with vertex set $R=\{v_{k+1},\cdots,v_n\}$. Let $v_j$ be the root of the component of $\Gamma_R$ containing $v_{k+1}$. By Lemma \ref{lem: Gamma(S)}(c), none of the edges of $\Gamma_R$ can pass through the vertical line at $v_j$. So, the edge $v_k-v_j$ is compatible with $\Gamma_R$. By maximality of $S$, we must have $v_k-v_j\in S$, as drawn. If $\Gamma_R$ is not connected, then, by induction on the size of $S$, we could add compatible edges in $E(R)$ to make $\Gamma_R$ connected. These compatible edges cannot cross the vertical line at $v_j$. So, they will be compatible with the edge $v_k-v_j$. So, they will be compatible with every element of $S$ contradicting its maximality.
\end{eg}
}

\begin{proof}[Proof of Theorem \ref{thm: maximal compatible sets are rooted binary trees}] (This argument is also the beginning of the proof of Theorem \ref {thm1: morphisms are maximal compatible sets of edge vectors}.)
{\color{blue}
Let $S$ be a maximal compatible subset of $G(V)$ and let $S_0=S\cap E(V)$. If $S=S_0$ then, by Lemma \ref{lem: Gamma(S)} (c), we have a root $v_k$ of $S_0$ so that $\ast-v_k$ is compatible with $S_0$. By maximality of $S$, we must have $S=S_0\cup \{\ast-v_k\}$.

We now claim that the graph $\Gamma(S_0)$ is connected with unique root $v_k$. If not then, by removing $v_k$ and the edge ending in $v_k$ we would have two subgraphs $\Gamma_L$ and $\Gamma_R$ as in Example \ref{eg: when Gamma(S) is not connected} and both subgraphs have to be connected (or empty) by maximality of $S$. See the argument given in Example \ref{eg: when Gamma(S) is not connected}. Therefore, $\Gamma(S_0)$ is a rooted binary tree. So, $S_0$ has $n-1$ elements and $S$ has $n$ elements.}
\end{proof}

Since we now know that maximal compatible subsets of $G(V)$ are augmented binary tree, the argument in this proof can be rephrased as follows. For any $X\in G(V)$ we use the notation $G_X(V)$ for the set of all elements of $G(V)$ which are compatible with $X$ {\color{blue}but not equal to $X$.}

\begin{cor}\label{cor: parts compatible with a root vector}
Let $R=\ast-v_k$ be a root vector in $G(V)$ where $V=\{v_1,\cdots,v_n\}$.
\begin{enumerate}
\item If $k=1$ or $n$ then there is a bijection
\[
	\widetilde\mu_R: G(V\backslash \{v_k\})\to G_R(V)
\]
given by $\widetilde\mu_R(E)=E$ for all $E\in E(V\backslash R)$ and $\widetilde\mu_R(\ast-v_j)=v_k-v_j$. Furthermore, $X,Y\in G(V\backslash \{v_k\})$ are compatible if and only if $\widetilde\mu_R(X),\mu_R(Y)$ are compatible in $G_R(V)$. When $k=n=1$, $\widetilde\mu_R$ is a bijection between two empty sets.
\item If $1<k<n$ then there is a bijection
\[
	\widetilde\mu_R:G(\{v_1,\cdots,v_{k-1}\})\smallcoprod G(\{v_{k+1},\cdots,v_n\})\to G_R(V)
\]
given by $\widetilde\mu_R(E)=E$ for all edges $E$ and $\widetilde\mu_R(\ast-v_j)=v_k-v_j$ for all root vectors $\ast-v_j$. Furthermore, $\widetilde\mu_R(X),\widetilde\mu_R(Y)\in G_R(V)$ are compatible if and only if {\color{blue}either $X,Y$ are compatible elements of the same block or} $X,Y$ lie in different blocks
\end{enumerate}
\end{cor}

This can be rephrased as follows. \vs2

\emph{$\widetilde\mu_R(X)$ is the unique element of $G_R(V)$ congruent to $X$ modulo $R=\ast-v_k$.}

\subsection{Proof of Theorem \ref{thm1: morphisms are maximal compatible sets of edge vectors}}

Let $\cR,\cS$ be noncrossing partitions of $V$ and suppose that $\cS$ is a refinement of $\cR$. {\color{blue}We take $\cR$, $\cS$ as in Example \ref{eg: composition of morpisms} as a simple example.} Then $E(\cS,\cR)$ is a disjoint union of \emph{blocks} $\cB(S_{\alpha\beta})$ defined as follows.

Let $\pi:\cS\to\cR$ be the mapping which sends each part of $\cS$ to the unique part of $\cR$ which contains it. For each $W_\alpha\in \cR$ let $\cU_\alpha=\pi^{-1}(W_\alpha)\subseteq \cS$. Then $\cU_\alpha$ is a noncrossing partition of $W_\alpha\subseteq V$. So, $\cU_\alpha$ is a disjoint union of parallel sets $S_{\alpha\beta}$ of which one, say $S_{\alpha0}$, is maximal. Every other $S_{\alpha\beta}$ is covered by some part $C_{\alpha\beta}\in \cU_\alpha$. {\color{blue}In the example, we have $W_\alpha=W$ and $\cU_\alpha=\{C,D,E\}$ which has two parallel sets: the maximal one $S_{\alpha 0}=\{C,E\}$ and $S_{\alpha1}=\{D\}$. This second one is covered by $C_{\alpha1}=C$.}

\begin{defn}
We define the \emph{maximal block} associated to the maximal parallel set $S_{\alpha0}$ to be the set $\cB(S_{\alpha0})=E(S_{\alpha0})\subseteq E(\cS,\cR)$. For $\beta\neq0$, define the \emph{block} $\cB(S_{\alpha\beta})$ associated to the parallel set $S_{\alpha\beta}$ covered by $C_{\alpha\beta}$ to be the image of the embedding 
\[
	\psi:G(S_{\alpha\beta})\into E(\cS,\cR)
\]
which sends each root vector $\ast-X$ to the edge $C_{\alpha\beta}-X$ and which is the inclusion map on $E(S_{\alpha\beta})$.
\end{defn}

{\color{blue}
In the example, the blocks are $\cB(S_{\alpha0})=E(\{C,E\})=\{C-E,E-C\}$ and $\cB(D)=\{C-D\}$ since $\psi(\ast-D)=C_{\alpha1}-D=C-D$.}

{\begin{lem}
{\color{blue}(a)} $E(\cS,\cR)$ is a disjoint union of blocks. 

{\color{blue}(b)} Elements of different blocks are always compatible. 

{\color{blue}(c)} Two elements of the same block are compatible if and only if the corresponding elements of $G(S_{\alpha\beta})$ (resp. $E(S_{\alpha0})$) are compatible as pieces of augmented (resp. unaugmented) binary trees on the  parallel set.
\end{lem}}
\begin{proof}
This follows from the definitions of the {\color{blue}terms. For (a), $E(\cS,\cR)=\coprod E(W_\alpha)$. Each $E(W_\alpha)$ is the disjoint union of $E(S_{\alpha\beta})$ for each of the parallel sets $S_{\alpha\beta}$ plus, for every $X$ in every nonmaximal parallel set $S_{\alpha\beta}$, the edge $C_{\alpha\beta}-X$ where $C_{\alpha\beta}\in W_\alpha$ is the single element which covers the elements of $S_{\alpha\beta}$. But these are, by definition, the blocks $\cB(S_{\alpha\beta})$.

(b) Every element of every block belongs either to a binary tree for the parallel set in the block or, for nonmaximal parallel sets, is equal to the ``covering edge'' $C-R$ where $R$ is the root and $C$ is the object which covers the root (and all other objects in the same parallel set). The union of these binary trees, together with the covering edges $C-R$, form one cluster morphism $\cR\to \cS$. Since the binary trees are chosen independently, elements from different blocks are always compatible: They belong to the same cluster morphism.

(c) The word ``compatibility'' is being used in two different ways. Elements of a block are compatible if they belong to the same cluster morphism. Definition \ref{def: compatibility for G(V)} is used for compatibility of elements of $G(S_{\alpha\beta})$. These definitions agree by the discussion in (a), (b) above.}
\end{proof}

\begin{proof}[Proof of Theorem \ref {thm1: morphisms are maximal compatible sets of edge vectors}]
This follows immediately from Theorem \ref {thm: maximal compatible sets are rooted binary trees} and the lemma above. Indeed a cluster morphism is defined to be a binary tree structure on every parallel set together with the edge from the root of that binary tree to the part which covers the parallel set in the case the parallel set is not maximal. By Theorem \ref {thm: maximal compatible sets are rooted binary trees}, such structures are maximal compatible subsets of the associated block in $E(\cS,\cR)$. By the lemma, the maximal compatible subsets of $E(\cS,\cR)$ are disjoint unions of maximal compatible subsets of the blocks. Therefore, these maximal compatible sets are the cluster morphisms $\cR\to \cS$.
\end{proof}

\subsection{First case of Theorem \ref {thm2: bijection ET to E}}
{\color{blue} The rest of Section 2 is devoted to the proof of Theorem \ref {thm2: bijection ET to E}. The proof is by induction on the size of $T$ and, in this subsection, we prove the base case when $T$ has only one element $E$. We also assume that $\cQ$ has only one part. Thus, we will show that the linear map $\pi_\ast:\ZZ\cS\to \ZZ\cR$ induces a bijection
\[
	\mu_E: E_E(\cS)\cong E(\cR)
\]where $E_E(\cS)$ is the set of all elements of $E(\cS)$ which are compatible with $E$ but not equal to $E$. Furthermore, we will show that two elements of $E_E(\cS)$ are compatible if and only if the corresponding elements of $E(\cR)$ are compatible. This is Lemmas \ref{lem:root vector case} and \ref{base case of bijection ET to T}.}

Recall that, for $E\in G(V)$, $G_E(V)$ is the set of all $X\in G(V)$ which are compatible with $E$ {\color{blue}but not equal to $E$}. We analyzed the case when $E$ is a root vector in Corollary \ref{cor: parts compatible with a root vector}. This translates into one case of Theorem \ref {thm2: bijection ET to E}.

{
\begin{figure}[htbp]
\begin{center}
\begin{tikzpicture}
	\foreach \x in {0.55,2.55,4.55,5.55}
	{\draw[thick, rounded corners=.35cm] (\x,-.35) rectangle +(.9,.75);
	} 
	\draw (-1.5,.35) node{$\cR=$}
	(0,0) node{1}
	(1,0) node{$B$:2}
	(2,0) node{3}
	(3,.75) node{$Z=A\cup C$}
	(3,0) node{$D$:4}
	(4,0) node{5}
	(5,0) node{$E$:6}
	(6,0) node{$F$:7}
	(7,0) node{8};
	\draw[thick,rounded corners=.35cm,blue]  (-.4,1)--(-.4,-.4)--(.4,-.4)--(.4,.5)--(1.6,.5)--(1.6,-.4)--(2.4,-.4)--(2.4,.5)--(3.6,.5)--(3.6,-.4)--(4.4,-.4)--(4.4,.5)--(5.4,.5)--(6.6,.5)--(6.6,-.4)--(7.4,-.4)--(7.4,1)--cycle;
\end{tikzpicture}
\color{blue}
\caption{We use as example $\cS$ from Example \ref{eg: composition of morpisms} and $\cR$ shown above, given by combining $D$ with $A$ to form a new part $Z=A\cup D$.}
\label{Fig: example R with one big part Z=AuC}
\end{center}
\end{figure}
}

{
Suppose that $\cU=\{X_1,\cdots,X_n\}$ is a nonmaximal parallel set in a noncrossing partition $\cS$ and let $Y$ be the part of $\cS$ which covers $\cU$. {\color{blue}(In $\cS$ from Example \ref{eg: composition of morpisms} we have $\cU=\{B,C,E,F\}$ and $Y=A$.)} Let $\cR$ be the noncrossing partition obtained from $\cS$ by merging the parts $X_k,Y$ to get a new part $Z=X_k\cup Y$. Recall that $\pi_\ast:\ZZ\cS\to \ZZ\cR$ is the linear surjection sending $X_k$ and $Y$ to $Z$ and all other parts of $\cS$ to the same part in $\cR$. The kernel of $\pi_\ast$ is the set of all integer multiples of the vector $R=Y-X_k$. {\color{blue}(In the example shown in Figure \ref{Fig: example R with one big part Z=AuC}, $Y=A$, $X_k=C$ and $R=A-C$.)}
}

\begin{lem}\label{lem:root vector case} (root vector case) The linear map $\pi_\ast:\ZZ\cS\to \ZZ\cR$ induces a bijection $E_R(\cS)\cong E(\cR)$ and this bijection preserves the relation of compatibility and non-compatibility.
\end{lem}

\begin{proof}
When $X_k$ is merged with $Y$, the parallel set $\cU$ in $\cS$ is, in general, divided into two parallel sets {\color{blue}$\cU'$ and $\cU''$} in $\cR$. {\color{blue}In Figure \ref{Fig: example R with one big part Z=AuC}, $\cU'=\{B\}$ and $\cU''=\{E,F\}$.} There are exceptional cases when $k=1$ or $n$ or $n=1$ analogous to trees listed in Corollary \ref{cor: parts compatible with a root vector}.

The set $E(\cS)$ is a disjoint union of blocks and the block {\color{blue}$\cB(\cU)=\psi G(\cU)$ reduces to $G_R(\cU)\subset E_R(\cS)$ which is isomorphic to the disjoint union of $G(\cU')$ and $G(\cU'')$ as in Corollary \ref{cor: parts compatible with a root vector} and these are isomorphic to the corresponding blocks $\cB(\cU')$ and $\cB(\cU'')$ of $\cR$. In Figure \ref{Fig: example R with one big part Z=AuC}, $\cB(\cU')=\{Z-B\}$, $\cB(\cU'')=\{Z-E,Z-F,E-F,F-E\}$ and
\[
    G_R(\cU)=\{C-B,C-E,C-F,E-F,F-E\}.
\]
Since the basepoint $\ast\in G(\cU'')$ and $\ast\in G(\cU')$ both correspond to $Z\in \cR$, the bijection $\widetilde\mu_R:G(\cU')\coprod G(\cU'')\cong G_R(\cU)$ from Corollary \ref{cor: parts compatible with a root vector} gives a bijection $\mu_R:\cB(\cU')\coprod \cB(\cU'')\cong G_R(\cU)$ by the formula
\[
	\mu_R(Z-X_j)=\widetilde \mu_R(\ast-X_j)=X_k-X_j
\]and $\mu(X)=X$ for all other elements of $\cB(\cU')\coprod \cB(\cU'')$.} In other words, for $X\in G_R(\cU)$, $\mu_R(\pi_\ast(X))$ is congruent to $X$ modulo $R=Y-X_k$. This is equivalent to the statement that $\mu_R^{-1}=\pi_\ast$. 

Since $\mu_R$ takes compatible pairs of elements to compatible pairs of elements in the case of trees, $\pi_\ast=\mu_R^{-1}$ preserves compatibility relations for noncrossing partitions since they are defined in terms of compatibility in the blocks that they lie in.
\end{proof}

{\color{blue}
\begin{rem}\label{rem: commuting diagram for mu-R}
The following commuting diagram illustrates the relation between the bijections $\mu_R$ and $\widetilde\mu_R$ used in the proof above.
\[
\xymatrixcolsep{30pt}
\xymatrix{
	G(\cU')\coprod G(\cU'') \ar[d]_\cong^\psi \ar[r]^(.6){\widetilde\mu_R} &
	G_R(\cU)\ar[d]^=\\
	\cB(\cU')\coprod \cB(\cU'') \ar[r]^(.6){\mu_R}&
	G_R(\cU).
	}
\]
\end{rem}
}

Now we examine the case when $E=X_b-X_a$ is an edge between two parallel parts of a noncrossing partition. As in the root vector case, the statement follows from the corresponding statement about compatible sets of edges for augmented binary trees.

The basic idea is that, when we build up an augmented binary tree starting with a fixed edge $E=v_k-v_{k-1}$ of length one, this edge behaves like a single vertex which we will label $v_k$. An augmented binary tree $T_+$ on $V$ which contains $E$ is equivalent to an augmented binary tree $T_+'$ on $V'=\{v_1,\cdots,\widehat{v_{k-1}},\cdots,v_n\}$, We take $T_+$ to be the union of all elements of $T_+'$ except for those of the form $v_k-v_j$ for $j<k$ which we replace with $v_{k-1}-v_j\in T_+$. (So that $v_k$ does not have two left children in $T_+$.)

If $T_+$ contains a fixed edge $E=v_b-v_a$ or length $b-a\ge2$ then it will also contain a binary tree $T''$ on the set $V''=\{v_{a+1},\cdots,v_{b-1}\}$ plus an edge from the root of $T''$ to $v_a$ since, by the noncrossing condition, no edge of $T_+$ can go from inside the interval $(a,b)$ to outside $[a,b]$. We treat this edge as corresponding to the root vector for $T_+''$. The rest of $T_+$ will be equivalent to an augmented binary tree $T_+'$ on $V'=\{v_1,\cdots,v_{a-1},v_b,\cdots,v_n\}$ with any edge $v_b-v_j\in T_+'$ with $j<a$ replaced with $v_a-v_j\in T_+$:
\[
	v_b-v_j\in T_+'\mapsto v_a-v_j\in T_+.
\]

{%
\begin{figure}[htbp]
\begin{center}
\begin{tikzpicture}[blue,scale=.6]
\begin{scope}[xshift=-8cm,yshift=4cm]
\draw[fill] (3,0) circle[radius=3pt];
\draw[fill] (4,2) circle[radius=3pt];
\draw[fill] (5,0) circle[radius=3pt];
\draw[thick] (3,0)--(4,2)--(5,0);
\draw (3,0) node[left]{$v_3$};
\draw (4,2) node[right]{$v_4$};
\draw (5,0) node[right]{$v_5$};
\draw (2,2) node{$T'':$};
\end{scope}

\coordinate (R8) at (8,8);
\coordinate (A6) at (6,6);
\coordinate (A9) at (9,6);
\coordinate (B2) at (2,4);
\coordinate (B7) at (7,4);
\coordinate (B10) at (10,4);
\coordinate (C1) at (1,3);
\coordinate (C4) at (4,2);
\foreach \x in {R8,A6,A9,B7,B10,C1}
	\draw[fill] (\x) circle[radius=3pt];
\draw[thick] (C1)--(A6)--(B7)  (A6)--(R8)--(A9)--(B10);
\draw (R8) node[above]{$v_8$};
\draw (5.7,6) node[above]{$v_6$};
\draw (A9) node[right]{$v_9$};
\draw (B7) node[left]{$v_7$};
\draw (B10) node[left]{$v_{10}$};
\draw (C1) node[left]{$v_1$};
\draw (0,7) node{$T':$};
\end{tikzpicture}
\color{blue}
\caption{Take $T$ as in Example \ref{eg: a binary tree}. Let $E=v_6-v_2$. Then we get trees $T'$ and $T''$ as shown above on vertex sets $V'=\{v_1,v_6,v_7,v_8,v_9,v_{10}\}$ and $V''=\{v_3,v_4,v_5\}$. To reconstruct $T$ we add back the vertex $v_2$, the edge $E=v_6-v_2$ and replace edge $v_6-v_1$ with $v_2-v_1$ (so that $v_6$ does not have two left children). We also add the edge $v_2-v_4$ connecting $v_2$ to the root of $T''$. If $T',T''$ are augmented, we first replace $\ast\in T_+'$ with $v_2$, then proceed to reconstruct $T_+$.}
\label{Fig: two trees out of one}
\end{center}
\end{figure}
}

This analysis for one edge $E$ in one tree $T_+$ translates into the following correspondence for all possible components of the trees $T_+,T_+',T_+''$. We note that $T_+,T_+',T_+''$ are arbitrary augmented binary trees on their vertex sets. (I.e., starting with any $T_+$ containing $E$, we get $T_+',T_+''$ and conversely, starting with any $T_+',T_+''$ we get a unique $T_+$ containing $E$.)

To do these cases simultaneously it is helpful to remember that $G(V)$ has $n^2$ elements. In particular, $G(\emptyset)=\emptyset$.

\begin{lem}\label{lem: positive edge case}
Suppose that $E=v_b-v_a$ where $1\le a<b\le n$. Then there is a bijection
\[
	\widetilde\mu_E:G(V')\smallcoprod G(V'')\to G_E(V)
\]
where $V'=\{v_1,\cdots,v_{a-1},v_b,v_{b+1},\cdots,v_n\}$ and $V''=\{v_{a+1},\cdots,v_{b-1}\}$ given by sending each $E'\in G(V')$ to $\widetilde\mu_E(E')=E'\in G_E(V)$ except for edges of the form $v_b-v_j\in G(V')$ for $j<a$ where we take
\[
	\widetilde\mu_E(v_b-v_j)=v_a-v_j.
\]
For $G(V'')$, each edge $E''\in E(V'')$ is sent to itself and the root vector $\ast-v_k$ is sent to 
\[
    \widetilde\mu_E(\ast-v_k)=v_a-v_k.
\]
Furthermore, the bijection $\widetilde\mu_E$ has the property that $X,Y$ in the domain of $\widetilde\mu_E$ are compatible if and only if $\widetilde\mu_E(X),\widetilde\mu_E(X)$ are compatible in $G_E(V)$.
\end{lem}

The analogous statement for {\color{blue}the left pointing edge} $E=v_a-v_b$ is the following.

\begin{lem}\label{lem: negative edge case}
Suppose that $E=v_a-v_b$ where $1\le a<b\le n$. Then there is a bijection
\[
	\widetilde\mu_E:G(V')\smallcoprod G(V'')\to G_E(V)
\]
where $V'=\{v_1,\cdots,v_{a},v_{b+1},\cdots,v_n\}$ and $V''=\{v_{a+1},\cdots,v_{b-1}\}$ given by sending each $E'\in G(V')$ to $\widetilde\mu_E(E')=E'\in G_E(V)$ except for edges of the form $v_a-v_j\in G(V')$ for $j>b$ where we take
\[
	\widetilde\mu_E(v_a-v_j)=v_b-v_j.
\]
For $G(V'')$, each edge $E''\in E(V'')$ is sent to itself and the root vector $\ast-v_k$ is sent to 
\[
    \widetilde\mu_E(\ast-v_k)=v_b-v_k.
\]
Furthermore, the bijection $\widetilde\mu_E$ has the property that $X,Y$ in the domain of $\widetilde\mu_E$ are compatible if and only if $\widetilde\mu_E(X),\widetilde\mu_E(X)$ are compatible in $G_E(V)$.
\end{lem}

\begin{proof}
In the discussion before the statements we started with an arbitrary augmented binary tree on $V$ with a fixed edge $E$ and constructed augmented binary trees on $V',V''$ which were also arbitrary. Two compatible elements of $G_E(V)$ can be extended to a tree. Passing to the corresponding trees on $V',V''$, we see that the corresponding edges are compatible.
\end{proof}

These statements are equivalent to the following special case of Theorem \ref{thm2: bijection ET to E}. Suppose that $\cS$ is a refinement of $\cQ$. We recall that, any edge $E\in E(\cS,\cQ)$ has the form $E=X_b-X_a$ where either $X_b$ covers $X_a$ or $X_a,X_b$ form part of a parallel set $\cU=\{X_1,\cdots,X_m\}\subseteq \pi^{-1}(W)$ in the inverse image $\pi^{-1}(W)\subseteq\cS$ of one part $W\in \cQ$. Recall that, in the second case, $E(\cS,\cQ)$ contains a block $\cB$ isomorphic to $G(\cU)$ (or to $E(\cU)$ when $\cU$ is maximal).

{\color{blue}
We take as an example, $\cS$ from Example \ref{eg: composition of morpisms}. Let $\cQ=\{V\}$ the partition of $V$ with one part (called $\cP$ in Example \ref{eg: composition of morpisms}). We merge two parts of $\cS$ to form $\cR$ shown in Figure \ref{Fig: example of partition R with 5 parts}.
{
\begin{figure}[htbp]
\begin{center}
\begin{tikzpicture}
	\foreach \x in {0.55,2.55,4.55}
	{\draw[thick, rounded corners=.35cm] (\x,-.35) rectangle +(.9,.75);
	} 
	\draw (-1.5,.7) node{$\cR=$}
	(0,0) node{1}
	(1,0) node{$B$:2}
	(2,0) node{3}
	(4,.95) node{$Z=C\cup F$}
	(3,0) node{$D$:4}
	(4,0) node{5}
	(5,0) node{$E'$:6}
	(7,0) node{8}
	(6,0) node{7};
	\draw[thick,rounded corners=.35cm]  (-.4,2)--(-.4,-.4)--(.4,-.4)--(.4,1.5)--(4.6,1.5)--(5.4,1.5)--(6.6,1.5)--(6.6,-.4)--(7.4,-.4)--(7.4,2)--cycle;
	\draw (0,1.5) node{$A$};
\begin{scope}[xshift=1cm]
	\draw[thick,rounded corners=.35cm]  (.6,1.2)--(.6,-.4)--(1.4,-.4)--(1.4,.7)--(2.6,.7)--(2.6,-.4)--(3.4,-.4)--(3.4,.7)--(4.6,.7)--(4.6,-.4)--(5.4,-.4)--(5.4,1.2)--cycle;
	\end{scope}
\end{tikzpicture}
\color{blue}
\caption{This is an example of partition $\cR$ given by merging the two parts $C,F$ in the parallel set $\cU=\{B,C,E',F\}$ in $\cS$ from Example \ref{eg: composition of morpisms}. The part $E$ has been renamed $E'$ to avoid confusion.}
\label{Fig: example of partition R with 5 parts}
\end{center}
\end{figure}
}
}

In general, the partition $\cR$ obtained from $\cS$ by merging the parts $X_a,X_b$ together is noncrossing and $\cR$ is a refinement of $\cQ$. Recall that $\pi_\ast:\ZZ\cS\to \ZZ\cR$ is the linear surjection sending $X_a$ and $X_b$ to $Z=X_a\cup X_b$ and all other parts $Y\in\cS$ to the same part $Y\in\cR$. The kernel of $\pi_\ast$ is the set of all integer multiples of the vector $E=X_a-X_b$.

\begin{lem}\label{base case of bijection ET to T}
When $rk\,\cR=1+rk\,\cS$, the linear map $\pi_\ast:\ZZ\cS\to \ZZ\cR$ induces a bijection $E_E(\cS,\cQ)\cong E(\cR,\cQ)$ and this bijection preserves the compatibility relation.
\end{lem}

\begin{proof} Without loss or generality we may assume that $\cQ$ has only one part since parts of $\cS$ and $\cR$ in different parts of $\cQ$ are unrelated. Then $E(\cR,\cQ)=E(\cR)$ and $E_E(\cS,\cQ)=E_E(\cS)$. The case when $X_b$ covers $X_a$ was settled in Lemma \ref{lem:root vector case}. So, suppose $X_a,X_b$ are parallel.

{\color{blue}There is one block $\cB=\cB(\cU)\subset E(\cS)$ containing the edge $E$. Any other block} $\cB'\subseteq E(\cS)$ is a block of both $E_E(\cS)$ and $E(\cR)$ and $\pi_\ast(\cB')=\cB'$. Let $\coprod \cB'$ be the union of these other blocks in both sets. Then, it suffices to consider the complement of $\coprod\cB'$ in $E_E(\cS)$ and $E(\cR)$. {\color{blue}These complements are the subsets corresponding to the block $\cB(\cU)\cong G(\cU)$ where $\cU=\{X_1,\cdots,X_m\}$.} The corresponding subset of $E_E(\cS)$ is
\[
	G_E(\cU)=G(\cU)\cap E_E(\cS)=E_E(\cS)\backslash\smallcoprod \cB'.
\]
{\color{blue}As in Remark \ref{rem: commuting diagram for mu-R}, there is a commuting diagram relating the bijections $\widetilde\mu_E$ from Lemmas \ref{lem: negative edge case}, \ref{lem: positive edge case} to a bijection
\[
    \mu_E:\cB(\cU')\coprod \cB(\cU'')\cong G_E(\cU)\subset E_E(\cS)
\]
where $\cB(\cU')$ and $\cB(\cU'')$ are the blocks in $E(\cR)$ given by the parallel sets
\[
	\cU'=\{X_1,\cdots,X_{a-1},Z,X_{b+1},\cdots,X_m\} \text{  and  } \cU''=\{ X_{a+1},\cdots,X_{b-1}\}
\]
in $\cR$.} Furthermore, this bijection sends each edge $X-Y$ to the same edge $X-Y$ except in the case when one of the parts $X,Y$ is equal to $Z=X_a\cup X_b$ in which case it is replaced by $X_a$ or $X_b$ whichever produces an edge compatible with $E$. 

{\color{blue} In our example $\cU=\{B,C,E',F\}$, $\cU'=\{B,Z\}$, $\cU''=\{E'\}$ and $E=F-C$.
\[
    \cB(\cU')=\{A-B,A-Z,B-Z,Z-B \},\quad \cB(\cU'')=\{Z-E'\}
\]
The bijection of $\cB(\cU')\coprod \cB(\cU'')$ with 
\[
G_E(\cU)=\{ A-B,A-F, B-F,C-B,C-E'\}
\]
is given by replacing $Z$ with $C$ or $F$ whichever is allowed.
}

In all cases, the {\color{blue}bijection $\mu_E$ corresponding to the bijections $\widetilde \mu_E$} of Lemmas \ref{lem: positive edge case} and \ref{lem: negative edge case} takes elements of $E(\cR)$ to elements in their inverse image in $E_E(\cS)$. So, $\mu_E=\pi_\ast^{-1}$. Since $\mu_E$ preserves compatibility, so does $\pi_\ast$.
\end{proof}

\subsection{Proof of Theorem \ref {thm2: bijection ET to E}}

The proof of Theorem \ref {thm2: bijection ET to E} is now an easy induction on the difference in ranks $rk\,\cR-rk\,\cS$. By Lemma \ref{base case of bijection ET to T}, the theorem holds when this difference is 1. So, suppose that $\cS$ is a refinement of $\cR$ and $\cR$ is a refinement of $\cQ$ and $rk\,\cR-rk\,\cS=k\ge2$. 

Let $T=\{E_1,\cdots,E_k\}$ be a cluster morphism $\cR\to\cS$. Then $T\subset E(\cS,\cQ)$. Suppose that $E_1=Y-X$ and let $\cP$ be obtained from $\cS$ by fusing $X,Y$ together to a single part $Z=X\cup Y$. Then $\cP$ is a refinement of $\cR$ and, by Lemma \ref{base case of bijection ET to T}, {\color{blue}the linear mapping $\pi_\ast:\ZZ\cS\to\ZZ\cP$ induces bijections:
\[
	\pi^{\cQ}_\ast:E_{E_1}(\cS,\cQ)\cong E(\cP,\cQ),\quad \pi^{\cR}_\ast:E_{E_1}(\cS,\cR)\cong E(\cP,\cR)
\]
where the second bijection is the restriction of the first to $E_{E_1}(\cS,\cR)\subseteq E_{E_1}(\cS,\cQ)$.

Since $[T]:\cR\to\cS$ is a cluster morphism, $E_2,\cdots,E_k$ are compatible with $E_1$ and with each other. So, $E_2,\cdots,E_k\in E_{E_1}(\cS,\cR)\subseteq E_{E_1}(\cS,\cQ)$ which map to compatible elements $\pi_\ast^{\cR}(E_i)\in E(\cP,\cR)\subseteq E(\cP,\cQ)$. Since $rk\,\cP-rk\,\cR=k-1$, this means that $S=\{\pi_\ast^{\cR}(E_2),\cdots,\pi_\ast^{\cR}(E_k)\}$ gives a cluster morphism $[S]:\cR\to \cP$. 

Since $\pi_\ast^{\cQ}$ sends compatible elements of $E_{E_1}(\cS,\cQ)$ to compatible elements of $E(\cP,\cQ)$, it induces a bijection
\[
	E_T(\cS,\cQ)\cong E_S(\cP,\cQ)
\]
By induction on $k$, the linear map $\rho_\ast:\ZZ\cP\to\ZZ\cR$ induces a bijection
\[
	\rho_\ast^{\cQ}:E_S(\cP,\cQ)\cong E(\cR,\cQ)
\]
Composing these we get a bijection
\[
	(\rho\pi)_\ast^\cQ:E_T(\cS,\cQ)\cong E(\cR,\cQ)
\]
induced by the composite linear mapping $\rho_\ast\circ\pi_\ast:\ZZ\cS\to \ZZ\cP\to\ZZ\cR$. Since $\pi_\ast^{\cQ}$ and $\rho_\ast^{\cQ}$ both preserve compatibility, so does their composite $(\rho\pi)_\ast^\cQ$.}

This proves Theorem \ref {thm2: bijection ET to E} for all $k$ and complete the proof that composition of cluster morphisms is well-defined and associative.

\section{Cube complexes}
\newcommand{\rank}{rank}

We will show that the classifying space $B\cN\cP(n)$ of the category of noncrossing partitions of $n$ is a locally $CAT(0)$ cube complex, i.e., its universal covering space is a $CAT(0)$ cube complex. Since $CAT(0)$ spaces are contractible, this implies that $B\cN\cP(n)$ is a $K(\pi,1)$. This follows from general sufficient conditions on a cubical category which we will describe.

\subsection{Cubical categories} The basic model for a category whose geometric realization is an $n$ cube is the category $\cI^n$ where $\cI$ is the poset category with two objects $0,1$ and one nonidentity morphism $0\to 1$ corresponding to the relation $0<1$. Recall that a poset $P$ can be taken to be the set of objects of a category $\cC(P)$ with a single morphism $A\to B$ whenever $A\le B$ and no morphisms $A\to B$ otherwise. 

The product category $\cI^n$ is also a poset category. The set of objects is the set of all vectors $x\in\RR^n$ so that every coordinate is either 0 or 1. We have $x\le y$ and thus a unique morphism $x\to y$ iff $x_i\le y_i$ for $i=1,\cdots,n$. It is very easy to see that the geometric realization of the category $\cI^n$ is canonically homeomorphic to the $n$-cube $[0,1]^n$. For example, for $n=2$ we have:
\begin{center}
\begin{tikzpicture}
\draw(0,0) node{ 
$
\xymatrixrowsep{5pt}\xymatrixcolsep{20pt}
\xymatrix{
 &(0,1)\ar[r] & (1,1)\\
 \cI^2: &&&& B\cI^2:\\
&(0,0)\ar[r]\ar[uu]\ar[uur] & (1,0)\ar[uu]
	}
$
};
\begin{scope}[xshift=4.8cm,yshift=-10mm]
\draw[fill,gray!35!white] (0,0) rectangle (2,2);
\foreach \x in {(0,0), (2,0), (0,2), (2,2)}
\draw[fill] \x circle[radius=2pt];
\draw[thick] (0,0)--(2,0)--(2,2)--(0,0)--(0,2)--(2,2);
\end{scope}

\end{tikzpicture}
\end{center}
We will define a ``cubical category'' to be a graded category in which the factorization category of any morphism of degree $n$ is isomorphic to $\cI^n$ for some $n$ plus other conditions. To be consistent with the rest of the paper we will use the word ``rank'' instead of ``degree.''

{\color{blue}
There are two factorizations of the rank 2 morphism $(0,0)\to (1,1)$. They are $(0,0)\to (0,1)\to (1,1)$ and $(0,0)\to (1,0)\to (1,1)$.

\begin{eg}\label{eg: 5 binary trees with two edges}
In Figure \ref{Fig: factorization of morphism of rank 2}, we use the notation of Definition \ref{def: binary tree notation} to specify 2 out of the 5 binary trees giving the cluster morphisms $\cP=\{V\}\to \Omega=\{A,B,C\}$ where $A,B,C=\{0\}, \{1\}, \{2\}$ and $V=\{0,1,2\}$. Each of these morphism has 2 factorizations giving the two factorization square shown.
\end{eg}
}

{
\begin{figure}[htbp]
\begin{center}
\begin{tikzpicture}
\begin{scope}[xshift=5.5cm,yshift=3cm,scale=.75]
\draw (0,0) node{0} circle[radius=.4cm];
\draw (1,0) node{1} circle[radius=.4cm];
\draw (2,0) node{2} circle[radius=.4cm];
\draw (-.9,0) node{$\Omega=$};
\end{scope}
\draw[thick,->] (.8,.4)--(.8,2.6); 
\draw[thick,->] (5.8,.4)--(5.8,2.6);
\draw[thick,->] (-4.3,.4)--(-4.3,2.6);
\draw[thick,->] (2.2,3)--(4.1,3); 
\draw[thick,->] (-.7,3)--(-2.8,3);
\draw[thick,->] (-.7,0)--(-2.8,0);
\draw[thick,->] (2.3,0)--(4.3,0);
\draw (.8,1.5) node[right]{$[\beta_{12}]$};
\draw (5.8,1.5) node[right]{$[-\beta_{01}]$};
\draw (-5.4,1.5) node[right]{$[\beta_{12}]$};
\draw (3.2,3) node[above]{$[\beta_{02}]$};
\draw (3.2,0) node[above]{$[\beta_{02}]$};
\draw (-1.8,3) node[above]{$[-\beta_{02}]$};
\draw (-1.8,0) node[above]{$[-\beta_{01}]$};
\begin{scope}[xshift=2.8cm,yshift=.9cm,scale=.8]
\coordinate (v0) at (0,1);
\coordinate (v1) at (1,0.5);
\coordinate (v2) at (2,1.5);
\foreach \x in {v0,v1,v2}
	\draw[fill] (\x) circle[radius=2pt];
\draw[thick] (v1)--(v0)--(v2);
\draw (v0) node[left]{0};
\draw (v1) node[right]{1};
\draw (v2) node[right]{2};
\draw (.2,.7) node[below]{\small$-\beta_{01}$};
\draw (.9,1.6) node{\small$\beta_{02}$};
\end{scope}
\begin{scope}[xshift=-2.4cm,yshift=1cm,scale=.8]
\coordinate (v0) at (0,1.5);
\coordinate (v1) at (1,0.5);
\coordinate (v2) at (2,1);
\foreach \x in {v0,v1,v2}
	\draw[fill] (\x) circle[radius=2pt];
\draw[thick] (v1)--(v2)--(v0);
\draw (v0) node[left]{0};
\draw (v1) node[left]{1};
\draw (v2) node[right]{2};
\draw (1.8,.3) node{\small$\beta_{12}$};
\draw (1.1,1.7) node{\small$-\beta_{02}$};
\end{scope}
\begin{scope}[xshift=-5cm,yshift=3cm,scale=.75]
\draw (0,0) node{0} circle[radius=.4cm];
\draw (1,0) node{1} circle[radius=.4cm];
\draw (2,0) node{2} circle[radius=.4cm];
\draw (-.9,0) node{$\Omega=$};
\end{scope}
\begin{scope}[xshift=-5cm,yshift=0cm,scale=.75]
\draw (0,0) node{0} circle[radius=.4cm];
\draw (1,0) node{1} ;
\draw (2,0) node{2} ;
\draw[rounded corners=.3cm]  (.6,.4)--(.6,-.4)--(2.4,-.4)--(2.4,.4)--cycle;
\end{scope}
\begin{scope}[xshift=5cm,yshift=0cm,scale=.75]
\draw (0,0) node{0};
\draw (1,0) node{1} ;
\draw (2,0) node{2} circle[radius=.4cm];
\draw[rounded corners=.3cm]  (-.4,.4)--(-.4,-.4)--(1.4,-.4)--(1.4,.4)--cycle;
\end{scope}
\begin{scope}[xshift=0cm,yshift=0cm,scale=.75]
\draw (0,0) node{0};
\draw (1,0) node{1} ;
\draw (2,0) node{2} ;
\draw[rounded corners=.3cm]  (-.4,.4)--(-.4,-.4)--(2.6,-.4)--(2.6,.4)--cycle;
\end{scope}
\begin{scope}[xshift=0cm,yshift=3cm,scale=.75]
\draw (0,0) node{0} ;
\draw (1,0) node{1} circle[radius=.4cm];
\draw (2,0) node{2} ;
\draw[rounded corners=.3cm]  (-.4,1.2)--(-.4,-.4)--(.4,-.4)--(.4,.7)--(1.6,.7)--(1.6,-.4)--(2.4,-.4)--(2.4,1.2)--cycle;
\end{scope}
\end{tikzpicture}
\color{blue}
\caption{Two morphisms of rank 2 from the one part partition $\cP=\{V\}$ of $V=\{0,1,2\}$ to the partition $\Omega$ of $V$ into 3 parts. Each morphism has 2!=2 factorizations, giving the two displayed factorization squares. Note that the only morphism allowed in the middle is $\beta_{12}$, not $-\beta_{12}$ since the part $W=\{0,2\}$ covers $1$. Figure \ref{Fig: composition of morphisms using trees} shows how the morphism on the right was computed.}
\label{Fig: factorization of morphism of rank 2}
\end{center}
\end{figure}
}

\begin{figure}[htbp]
\begin{center}
\begin{tikzpicture}
\begin{scope}[xshift=-5cm,]
\coordinate (v0) at (0,0);
\coordinate (v1) at (1,0);
\coordinate (v2) at (2,1);
\foreach \x in {v0,v1,v2}
	\draw[fill] (\x) circle[radius=2pt];
\draw[very thick] (v0)--(v2);
\draw[thick] (v2)--(v1);
\draw (v0) node[left]{0};
\draw (v1) node[left]{1};
\draw (v2) node[right]{2};
\draw (1.8,.3) node{\small$\beta_{12}$};
\draw (1,.9) node{\small$\beta_{02}$};
\end{scope}
\begin{scope}
\coordinate (v0) at (0,.5);
\coordinate (v1) at (1,0);
\coordinate (v2) at (2,1.5);
\foreach \x in {v0,v1,v2}
	\draw[fill] (\x) circle[radius=2pt];
\draw[thick] (v1)--(v0);
\draw[very thick] (v0)--(v2);
\draw (v0) node[left]{0};
\draw (v1) node[right]{1};
\draw (v2) node[right]{2};
\draw (.3,-.1) node{\small$-\beta_{01}$};
\draw (1,1.4) node{\small$\beta_{02}$};
\end{scope}
\end{tikzpicture}
\color{blue}
\caption{This shows how composition $[\beta_{02}]\circ[\beta_{12}]$ can be computed diagramatically. If we naively put these edges together we get the graph on the left side. This is not allowed since $v_2$ cannot have two left children. We are not allowed to change the second edge $\beta_{02}$. So, we have to change $\beta_{12}$. We are allowed to add a multiple of $\beta_{02}$. Visually this means we can slide the right hand vertex of $\beta_{12}$ down to the other end of $\beta_{02}$ giving the figure on the right. This one is allowed. So, it is the composition. (The two displayed graphs are the only possibilities.)}
\label{Fig: composition of morphisms using trees}
\end{center}
\end{figure}
\begin{defn}
Let $f:A\to B$ be a morphism in a category $\cC$. Then the \emph{factorization category} $Fac(f)$ is the category whose objects are \emph{factorizations} of $f$ given by triples $(C,g,h)$ where $C$ is an object of $\cC$ and $g,h$ are morphisms $g:A\to C$, $h:C\to B$ so that $f=h\circ g$. 
\[
	A\xrarrow g C\xrarrow h B
\]
A morphism $(C,g,h)\to (C',g',h')$ is defined to be a morphism $\phi:C\to C'$ so that $g'=\phi\circ g$ and $h=h'\circ\phi$. There is a forgetful functor $Fac(f)\to\cC$ taking $(C,g,h)$ to $C$. The morphism $f$ is \emph{irreducible} if, for any factorization $f=h\circ g$, either $g$ or $h$ is an isomorphism.

We call $g:A\to C$ a ``first factor'' of $f$ if $g$ is irreducible. We call $h:C\to B$ a ``last factor'' of $f$ if $h$ is irreducible.
\end{defn}

\begin{defn}
A \emph{cubical category} is a small category $\cC$ with the following properties.
\begin{enumerate}
\item Every morphism $f:A\to B$ in the category has a \emph{{\rank}} $rk\,f$ which is a nonnegative integer so that $rk\,(f\circ g)=rk\,f+rk\,g$.
\item If $rk\,f=n$ then the factorization category of $f$ is isomorphic to the standard $n$-cube category: $Fac(f)\cong\cI^n$.
\item The forgetful functor $Fac(f)\to\cC$ is an embedding. In particular every morphism of {\rank} $n$ has $n$ distinct first factors and $n$ distinct last factors.
\item Every morphism of {\rank} $n$ is determined by its $n$ first factors.
\item Every morphism of {\rank} $n$ is determined by its $n$ last factors.
\end{enumerate}
\end{defn}

Condition (1) implies that every isomorphism has rank 0. Condition (2) implies that every rank 0 morphism is an identity map. So, a cubical category has only one object in every isomorphism class.
\vs2

\begin{minipage}{0.7\textwidth}
\begin{eg}\label{eg: morphism with 6 factorizations} {\color{blue}Figure \ref{Fig: cube} shows the $6$ factorizations of the morphism of rank 3 given by the binary tree shown on the right. Example \ref{eg: 5 binary trees with two edges} gives easier examples. We always use $\Omega$ to denote the partition of a set $V$ into its individual elements.}
\end{eg}
\end{minipage}
\hfill
\begin{minipage}{0.3\textwidth}

\begin{center}
\begin{tikzpicture}
\coordinate (A) at (0,1);
\coordinate (B) at (1,0);
\coordinate (C) at (2,.5);
\coordinate (BC) at (1.7,.25);
\coordinate (D) at (3,1.5);
\draw[thick] (B)--(C)--(A)--(D);
\foreach \x/\xtext in {A/0,B/1,C/2,D/3}
\draw[fill] (\x) circle[radius=2pt];
\draw (A) node[left]{$0$};
\draw (B) node[left]{$1$};
\draw (C) node[right]{$2$};
\draw (D) node[right]{$3$};
\draw (BC) node[below]{$\beta_{12}$};
\draw (.6,.6) node{$-\beta_{02}$};
\draw (1.3,1.5) node{$\beta_{03}$};
\end{tikzpicture}
\end{center}
\end{minipage}

\vs3

\begin{figure}[htbp]
\begin{center}
\begin{tikzpicture}
\draw(0,0) node{ 
$ 
\xymatrixrowsep{30pt}\xymatrixcolsep{50pt}
\xymatrix{
 & \bullet\ar[rr]^{[\beta_{12}]} & & \bullet\\
 \bullet\ar[ru]^{[-\beta_{01}]}\ar[rr]^(.7){[\beta_{12}]} && \bullet\ar[ru]^(.4){[-\beta_{02}]}\\
 &\bullet\ar[uu]_(.3){[\beta_{03}]}\ar[rr]^(.3){[\beta_{12}]} && \bullet\ar[uu]_(.5){[\beta_{03}]}\\
 \bullet \ar[uu]^(.5){[\beta_{03}]}\ar[ru]^{[\beta_{13}]}\ar[rr]_{[\beta_{12}]} && \bullet\ar[ru]_{[\beta_{23}]}\ar[uu]^(.7){[\beta_{03}]}
	}
$
}; 
\begin{scope}[xshift=4.2cm,yshift=2.7cm,scale=.75]
\draw (0,0) node{0} circle[radius=.4cm];
\draw (1,0) node{1} circle[radius=.4cm];
\draw (2,0) node{2} circle[radius=.4cm];
\draw (3,0) node{3} circle[radius=.4cm];
\draw (-.9,0) node{$\Omega=$};
\end{scope}
\begin{scope}[xshift=-4cm,yshift=2.7cm,scale=.75]
\draw (0,0) node{0} circle[radius=.4cm];
\draw (1,0) node{1} ;
\draw (2,0) node{2} ;
\draw (3,0) node{3} circle[radius=.4cm];
\draw[rounded corners=.3cm]  (.6,.4)--(.6,-.4)--(2.4,-.4)--(2.4,.4)--cycle;
\end{scope}
\begin{scope}[xshift=-6.5cm,yshift=1cm,scale=.75]
\draw (0,0) node{0};
\draw (1,0) node{1} ;
\draw (2,0) node{2} ;
\draw (3,0) node{3} circle[radius=.4cm];
\draw[rounded corners=.3cm]  (-.4,.4)--(-.4,-.4)--(2.4,-.4)--(2.4,.4)--cycle;
\end{scope}
\begin{scope}[xshift=-6.5cm,yshift=-2.5cm,scale=.75]
\draw (0,0) node{0};
\draw (1,0) node{1} ;
\draw (2,0) node{2} ;
\draw (3,0) node{3} ;
\draw[rounded corners=.3cm]  (-.4,.4)--(-.4,-.4)--(3.4,-.4)--(3.4,.4)--cycle;
\end{scope}
\begin{scope}[xshift=4cm,yshift=-1cm,scale=.75]
\draw (0,0) node{0} ;
\draw (1,0) node{1} circle[radius=.4cm];
\draw (2,0) node{2} circle[radius=.4cm];
\draw (3,0) node{3} ;
\draw[rounded corners=.3cm]  (-.4,1.2)--(-.4,-.4)--(.4,-.4)--(.4,.7)--(2.6,.7)--(2.6,-.4)--(3.4,-.4)--(3.4,1.2)--cycle;
\end{scope}
\begin{scope}[xshift=-1.5cm,yshift=-1.5cm,scale=.5] 
\draw (0,0) node{0} ;
\draw (1,0) node{1} ;
\draw (2,0) node{2} ;
\draw (3,0) node{3} ;
\draw[rounded corners=.2cm]  (-.4,1.2)--(-.4,-.4)--(.4,-.4)--(.4,.7)--(2.6,.7)--(2.6,-.4)--(3.4,-.4)--(3.4,1.2)--cycle;
\draw[rounded corners=.2cm]  (.6,.4)--(.6,-.4)--(2.4,-.4)--(2.4,.4)--cycle;
\end{scope}
\begin{scope}[xshift=1.7cm,yshift=-3.2cm,scale=.75]
\draw (0,0) node{0} ;
\draw (1,0) node{1} circle[radius=.4cm];
\draw (2,0) node{2} ;
\draw (3,0) node{3} ;
\draw[rounded corners=.3cm]  (-.4,1.2)--(-.4,-.4)--(.4,-.4)--(.4,.7)--(1.6,.7)--(1.6,-.4)--(3.4,-.4)--(3.4,1.2)--cycle;
\end{scope}
\begin{scope}[xshift=1.5cm,yshift=.2cm,scale=.5]
\draw (0,0) node{0} ;
\draw (1,0) node{1} circle[radius=.4cm];
\draw (2,0) node{2} ;
\draw (3,0) node{3} circle[radius=.4cm];
\draw[rounded corners=.2cm]  (-.4,1.2)--(-.4,-.4)--(.4,-.4)--(.4,.7)--(1.6,.7)--(1.6,-.4)--(2.4,-.4)--(2.4,1.2)--cycle;
\end{scope}
\end{tikzpicture}\color{blue}
\caption{Cube showing the $3!=6$ factorizations of the morphism of rank 3 given by the binary tree with edge vectors $\beta_{12},-\beta_{02},\beta_{03}$ shown in Example \ref{eg: morphism with 6 factorizations} above.}
\label{Fig: cube}
\end{center}
\end{figure}

In a cubical category, a collection of {\rank} one morphisms $f_i:X\to Y_i$, $i=1,\cdots,k$ are called \emph{s-compatible} if they form the first factors of a (unique) {\rank} $k$ morphism $X\to Z$. Similarly, a collection of {\rank} one morphisms $g_i:Y_i\to Z$, $i=1,\cdots,k$ are called \emph{t-compatible} if they form the last factors of a (unique) {\rank} $k$ morphism $X\to Z$.

\begin{prop}\label{prop: good properties of C}
Suppose that $\cC$ is a cubical category. Then the following additional properties are sufficient for the classifying space $B\cC$ to be locally $CAT(0)$ and thus a $K(\pi,1)$.
\begin{enumerate}
\item Two unequal morphisms $f\neq g:X\to Y$ between any two objects give nonhomotopic paths from $X$ to $Y$ in $B\cC$.
\item A set of {\rank} 1 morphisms $f_i:X\to Y_i$ is s-compatible if and only if it is pairwise s-compatible.
\item A set of {\rank} 1 morphisms $g_i:Y_i\to Z$ is t-compatible if and only if it is pairwise t-compatible.
\end{enumerate}
\end{prop}

\begin{proof}
Suppose that $\cC$ is a cubical category. Then the classifying space of $B\cC=|\cN_\bullet\cC|$ is a union of cubes since every sequence of composable morphisms is contained in the factorization cube of its composition. As a simplicial complex, these cubes are triangulated. However, we will ignore this subdivision of each cube and consider the structure of $B\cC$ as a union of embedded cubes.

\underline{Claim 1}. Given Condition (1), the universal covering of $B\cC$ is a cube complex, i.e., the intersection of any two cubes is a common face.

Proof: Let $\widetilde C_1, \widetilde C_2$ be cubes in the universal cover $\widetilde{B\cC}$ of $B\cC$ which lie over the factorization cubes $C_i\subset B\cC$ of $f_i:X_i\to Y_i$. Let $\widetilde Z_1,\cdots,\widetilde Z_k$ be the vertices of $\widetilde C_1\cap \widetilde C_2$. These lie over $Z_1,\cdots,Z_k\in C_1\cap C_2$. It is clear that $\widetilde C_1\cap \widetilde C_2$ is a union of cubes. Thus, it suffices to show that there is a common face $C_0$ of $C_1, C_2$ which contains all the $Z_i$. This would imply $\widetilde C_1\cap \widetilde C_2=\widetilde C_0$, a common face.

Consider the maps $g_i:X_1\to Z_i$ and $h_i:Z_i\to Y_2$ in the factorizations of $f_1,f_2$. These lift to maps $\widetilde g_i:\widetilde X_1\to \widetilde Z_i$ and $\widetilde h_i:\widetilde Z_i\to \widetilde Y_2$ in $\widetilde C_1,\widetilde C_2$. Then the compositions $h_ig_i:X_1\to Y_2$ lift to paths $\widetilde X_1\to \widetilde Y_2$ with the same endpoints in $\widetilde{B\cC}$. So, they are homotopic in $B\cC$. By Condition (1), these morphisms $h_ig_i:X_1\to Y_2$ are equal to the same morphism $f_3:X_1\to Y_2$. Let $C_3=Fac(f_3)$. 

The first factors of the $g_i:X_1\to Z_i$ together determine a morphism $f_1':X_1\to Y_1'$ which is part of the factorization of both $f_1$ and $f_3$. So $C_1'=Fac(f_1')$ is a common face of $C_1,C_3$ which contains all the $Z_i$. Similarly, the last factors of $h_i:Z_i\to Y_2$ together give a morphism $f_2':X_2'\to Y_2$ factoring both $f_2,f_3$. So, $C_2'=Fac(f_2')$ is a common face of $C_2,C_3$ containing all the $Z_i$. Then $C_1'\cap C_2'$ is a common face of $C_1,C_2,C_3$ containing all the $Z_i$ and we are done.

By Gromov's Theorem \ref{thm of Gromov}, it now suffices to show that the link of ever vertex of $B\cC$ is flag. For any object $X$ in $\cC$ we define the \emph{forward link} $Lk_+(X)$ of $X$ to be the simplicial complex whose vertices are all rank one morphisms with source $X$ so that a collection of such morphisms forms a simplex if they form the first factors of some morphism in $\cC$. 

By Condition (2), $Lk_+(X)$ is a flag complex for all objects $X\in \cC$. The \emph{backward link} $Lk_-(X)$ of $X$ is defined analogously. By Condition (3), $Lk_-(X)$ is also a flag complex.

Finally, we observe that, for any object $X$ in $\cC$, the link of the vertex $X$ in $B\cC$ is isomorphic as a simplicial complex to the join of the forward link of $X$ and its backward link. Consequently, it is a flag complex. So, the universal cover $\widetilde{B\cC}$ of $B\cC$ is a $CAT(0)$ cubical complex as claimed.
\end{proof}

\begin{rem}
We have already shown that $\cN\cP(n)$ satisfies Condition (3). Indeed, when a morphism $[T]:\cP\to\cQ$ of rank, say $k$, is decomposed into a product of rank one morphisms, the last morphism $\cR\to\cQ$ is, by definition, given by one of the $k$ elements of $T\subseteq E(\cQ)$. By Theorem \ref{thm1: morphisms are maximal compatible sets of edge vectors}, such elements of $E(\cQ)$ are compatible if and only if they are pairwise compatible.
\end{rem}

We restate condition (1) in Proposition \ref{prop: good properties of C} in a more convenient format. 

\begin{defn}\label{def: faithful group functor} We define a \emph{faithful group functor} on a category $\cC$ to be a faithful functor $\varphi:\cC\to G$ where $G$ is a group considered as a groupoid with one object.
\end{defn}

To help clarify the definition we note that a {functor} $\varphi:\cC\to G$ assigns a group element $\varphi(f)\in G$ to every morphism $f:X\to Y$ in $\cC$ so that $\varphi(fh)=\varphi(f)\varphi(h)$ for any pair of composable morphisms in $\cC$. This functor is \emph{faithful} if $f\neq h$ implies $\varphi(f)\neq \varphi(h)$. {\color{blue}This is directly related to Condition (1) in Proposition \ref{prop: good properties of C} since a group functor $\varphi:\cC\to G$ will assign an element of $G$ to any path in the category. For example, take the path $X\xrightarrow{f} Y \xleftarrow{g}Z\xrightarrow h W$. We can assign to this path the group element $\varphi(h)\varphi(g)^{-1}\varphi(f)$. Any two homotopic paths will have the same associated group element. So, $f\neq g:X\to Y$ cannot be homotopic since, if they were, they must have the same group element attached $\varphi(f)=\varphi(g)$ contradicting faithfulness of $\varphi$. Thus, the existence of a faithful group functor on $\cC$ implies that $\cC$ satisfies Condition (1). In Proposition \ref{prop: condition (1) is equivalent to group functor} below we show the stronger statement that Condition (1) is equivalent to the existence of a faithful group functor on $\cC$.}

Since a connected groupoid is equivalent to a groupoid with one object, the existence of a faithful group functor on $\cC$ is equivalent to the existence of a faithful functor of $\cC$ into a connected groupoid.

\begin{prop}\label{prop: condition (1) is equivalent to group functor}
A small connected category $\cC$ satisfies Condition (1) in Proposition \ref{prop: good properties of C} if and only if it admits a faithful group functor $\varphi:\cC\to G$ for some group $G$.
\end{prop}

\begin{proof}
We recall that the \emph{fundamental groupoid} of $\cC$ is another category $\pi_1\cC$ with the same object set as $\cC$ but where morphisms from $X$ to $Y$ are defined to be homotopy classes of paths from $X$ to $Y$ in the classifying space $B\cC$. Since any morphism $f:X\to Y$ gives a path $X\to Y$ whose homotopy class is a morphism $\pi_1 f:X\to Y$ in $\pi_1\cC$, there is a functor 
$	\pi_1:\cC\to\pi_1\cC$. 
Furthermore, this functor is universal among all functors of $\cC$ into all groupoids. Condition (1) in Proposition \ref{prop: good properties of C} is clearly equivalent to the statement that this functor is faithful. As noted above, the existence of a faithful group functor on $\cC$ is equivalent to the existence of a faithful functor of $\cC$ into a connected groupoid. {\color{blue}So, Condition (1) is equivalent to the existence of a faithful group functor on $\cC$.}
\end{proof}

\subsection{Representation of $\cN\cP(n)$}

There is a ``standard representation'' of the category $\cN\cP(n)$ given as follows. Let $U_n(\ZZ)$ be the group of $n\times n$ unipotent upper triangular matrices with integer entries. 

\begin{prop}\label{existence of matrix gT} For every morphism $[T]:\cP\to\cQ$ between any two objects of $\cN\cP(n)$, there is a matrix
$
	\varphi[T]\in U_n(\ZZ)
$
with the following properties.
\begin{enumerate}
\item $\varphi([T]\circ[S])=\varphi[S]\varphi[T]$. {\color{blue}The order reversal implies $\varphi[T]^{-1}$ is a group functor. We say that $\varphi$ is a ``contravariant group functor''.}
\item If $[S]\neq[T]:\cP\to\cQ$ then $\varphi[T]\neq \varphi[S]$.
\end{enumerate}
\end{prop}

\begin{cor} $\varphi:\cN\cP(n)\to U_n(\ZZ)$ is a faithful (contravariant) group functor and, therefore, $\cN\cP(n)$ satisfies Condition (1) of Proposition \ref{prop: good properties of C}.\qed
\end{cor}

The remainder of this subsection is devoted to the proof of Proposition \ref{existence of matrix gT}.

 \begin{defn}
 Recall that a morphism $[T]:\cP\to\cQ$ is given by a rooted tree $T_\alpha$ for every part $W_\alpha\in \cP$. This gives a partial ordering on the subset $\cQ_\alpha\subseteq\cQ$ of all parts which lie in $W_\alpha$. Let $\varphi[T]\in U_n(\ZZ)$ be given by  {\color{blue}letting $\varphi_{ij}[T]$, the $ij$-entry of the matrix $\varphi[T]$, be equal to 1} if either $v_i=v_j$ or all the following hold.
 \begin{enumerate}
 \item $v_i,v_j$ lie in the same part of $\cP$, say $v_i,v_j\in W_\alpha$, but in different parts $X,Y$ of $\cQ_\alpha$,
 \item $i< j$ and $X<Y$ in the order given by the tree $T_\alpha$.
 \item $j$ is minimal so that $i<j$ and $v_j\in Y$
 \end{enumerate}
  and $\varphi_{ij}[T]=0$ otherwise.
 \end{defn}
 
Since $X,Y$ are noncrossing it follows that, if $\varphi_{ij}[T]=1$ and $i\neq j$, then $\varphi_{i'j}[T]=1$ for all other $v_{i'}\in X$, the part of $\cQ$ containing $v_i$.
 
 \begin{eg}
 Let $\cP,\cQ$ be the following partitions of 5: $\cP=\{V\}$, $\cQ=\{X,Y\}$ where $X=\{v_2,v_3\}$, $Y=\{v_1,v_4,v_5\}$ {\color{blue}and $V=X\cup Y$.} Then there is only one morphism $[S]:\cP\to\cQ$ given by the tree on $\cQ$ with root $Y$. Then
 \[
 	\varphi[S]=\mat{
	1&0&0&0&0\\
	&1 &0& 1&0\\
	&&1 &1&0\\
	&&&1&0\\
	&&&&1
	}
 \]
 The nonzero nondiagonal entries are $\varphi_{24}=\varphi_{34}=1$ since $j=4$ is the smallest index of any element of $Y$ to the right of $X$. Thus left multiplication by $\varphi[S]$ is the row operation which adds Row 4 to Row 2 and to Row 3. Let $\Omega$ be the noncrossing partition of $V$ into its five individual elements. Let $[T]:\cQ\to\Omega$ be the morphism given by the binary trees on $V_\alpha=\{2,3\}$ and $V_\beta=\{1,4,5\}$ given by $v_2<v_3$ and $v_1<v_4<v_5$. Then
 \[
  	\varphi[T]=\mat{
	1&0&0&1&1\\
	&1 &1& 0&0\\
	&&1 &0&0\\
	&&&1&1\\
	&&&&1
	},\qquad 	
	  	\varphi[S]\,\varphi[T]=\mat{
	1&0&0&1&1\\
	&1 &1& 1&1\\
	&&1 &1&1\\
	&&&1&1\\
	&&&&1
	}
 \]
 \begin{minipage}{0.6\textwidth}

The composition $[T]\circ [S]=[T\cup\mu_TS]$ is given by adding the edge $-\beta_{13}=v_1-v_3$ to $T$ shown on the right. The partial ordering of vertices given by this new tree is
$
 	v_2<v_3<v_1<v_4<v_5
 $
 giving the matrix $\varphi[S]\varphi[T]$ as required {\color{blue}making $\varphi$ a contravariant group functor.}

\end{minipage}
\hfill
\begin{minipage}{0.4\textwidth}

\begin{center}
\begin{tikzpicture}
\coordinate (E) at (4,1.72);

\coordinate (A) at (0,1);
\coordinate (B) at (1,0);
\coordinate (C) at (2,.5);
\coordinate (BC) at (1.7,.25);
\coordinate (D) at (3,1.54);
\draw[thick] (B)--(C)--(A)--(D)--(E);
\foreach \x/\xtext in {A/1,B/2,C/3,D/4,E/5}
\draw[fill] (\x) circle[radius=2pt];
\draw (A) node[left]{$1$};
\draw (B) node[left]{$2$};
\draw (C) node[right]{$3$};
\draw (D) node[below]{$4$};
\draw (E) node[below]{$5$};
\draw (.6,.6) node{$-\beta_{13}$};
\draw (-.4,1.7) node{$T:$};
\end{tikzpicture}
\end{center}
\end{minipage}
  \end{eg}
 
 \begin{lem}\label{lem: U is faithful}
 A morphism $[T]:\cP\to\cQ$ is uniquely determined by $\cP,\cQ$ and the matrix $\varphi[T]$.
 \end{lem}
 
 \begin{proof}
 The tree $T$ is a union of augmented binary trees on relative parallel sets. Since the augmentation of a binary tree is unique, it suffices to determine the partial ordering and thus the binary tree on every parallel set in $\cP$ relative to $\cQ$. So, let $\cU=\{X_1,\cdots,X_n\}$ be a parallel set in $\cQ_\alpha\subseteq \cQ$, the union of parts in $W_\alpha\in\cP$. For each $i$, let $v_i$ be the leftmost element of $X_i$ and consider only the entries $\varphi_{ij}[T]$ corresponding to the $n$ points $v_i,v_j$.\vs2
 
 \underline{Claim:} The matrix $\varphi[T]$ determines the root $X_k$ of $\cU$.\vs2
 
 Proof: Take $k$ maximal so that $\varphi_{ik}=1$ for all $i<k$. For any $j>k$ we have $\varphi_{jk}=0$. So, this holds only when $X_k$ is the root.\vs2
 
 By induction on $n$, the submatrices $\varphi_{ij}[T]$ for $i,j<k$ and $\varphi_{ij}[T]$ for $i,j>k$ uniquely determine the subtrees of the tree on $\cU$ given by deleting the root. Therefore, the matrix determines the entire binary tree on $\cU$ and therefore the entire rooted tree $T$.
 \end{proof}

 \begin{proof}[Proof of Proposition \ref{existence of matrix gT}] Lemma \ref{lem: U is faithful} proves Property (2). To prove Property (1), we first reduce it to the case when $S$ has only one element. Indeed, given composable morphisms
 \[
 	\cQ\xrarrow{[S]} \cR\xrarrow{[T]} \cS
 \]
where $S$ has more than one element, $[S]$ can be written as a composition of two shorter morphisms $[S]=[S_1]\circ[S_2]$. Then, by induction on $|S|$ we have
\[
	\varphi([T]\circ[S])=\varphi([T]\circ[S_1]\circ[S_2])=\varphi[S_2]\varphi([T]\circ[S_1])=\varphi[S_2]\varphi[S_1]\varphi[T]=\varphi[S]\varphi[T]
\]
So, it is enough to prove the base case when $S$ has one element, say $S=Y-X$. There are three cases. Either $X,Y$ are parallel with $Y$ to the left or right of $X$ or $Y$ covers $X$. 

In all cases we have $[T]\circ[S]=[T\cup \mu_TS]$ where $\mu_TS=Z-R\in E_T(\cS,\cQ)$ where $R$ is the root of the rooted tree $T_X$ given by restricting $T$ to the parts of $\cS$ which lie in $X\in\cR$ and $Z$ is a part of $\cS$ which lies in $Y$ and which depends on the case.
\vs2

\underline{Case 1}: Suppose $X,Y$ are parallel with $Y$ to the left of $X$. 

In this case $\varphi[S]=I_n$ is the identity matrix. The addition of the edge $\mu_TS$ to $T$ raises some parts in $Y$ above all parts of $X$. However, since $Y$ is to the left of $X$, this has no effect on the matrix. So,
\[
	\varphi[T\cup \mu_TS]=\varphi[T]=\varphi[S]\varphi[T].
\]

\underline{Case 2}: Suppose that $X,Y$ are parallel with $Y$ to the right of $X$.

Let $v_j$ be the leftmost element of $Y$. Then, by definition, $\varphi[S]$ is the identity matrix except for the $j$th column where $\varphi_{ij}[S]=1$ if and only if either $i=j$ or $v_i\in X$. The new edge $\mu_TS=Z-R$ goes from the root $R$ of $T_X$ to the leftmost part $Z\in\cS$ of $Y$. When this new edge is added, every part of $T_X$ becomes less than $Z$ (this makes $\varphi_{ij}=1$ for all $v_i\in X$) and thus, by transitivity, less than any other part which is greater than $Z$. In other words, $\varphi[T\cup\mu_TS]$ is obtained from $\varphi[T]$ by adding its $j$th row to the rows corresponding to the elements of $X$. So, $\varphi[T\cup \mu_TS]=\varphi[T]=\varphi[S]\varphi[T]$ in this case as well.\vs2

\underline{Case 3}: Suppose $Y$ covers $X$.

In this case $Z$ is the part of $\cS$ which contains the leftmost element $v_j$ of $Y$ to the right of $X$. As in the other cases, $\varphi[S]$ is equal to the identity matrix except in Column $j$ where $\varphi_{ij}[S]=1$ iff either $i=j$ or $v_i\in X$. Again, the addition of $\mu_TS=Z-R$ makes $Z$ greater than all parts of $X$ in $\cR$ and thus changes $\varphi[T]$ by adding Row $j$ to Row $i$ for every $v_i\in X$ which is the same as left multiplication by $\varphi[S]$.\vs2

This proves Property (1) in all three cases and completes the proof of Proposition  \ref{existence of matrix gT}.
 \end{proof}
 
\subsection{$\cN\cP(n)$ is cubical}
 
Let $[T]:\cP\to\cQ$ be a cluster morphism of rank $k$.

\begin{lem}
The factorization category $Fac[T]$ is isomorphic to the cube $\cI^k$ and the forgetful functor $Fac[T]\to\cN\cP(n)$ is an embedding.
\end{lem}

\begin{proof} Let $T=\{E_1,\cdots,E_k\}$. Then any factorization of $[T]$ has the form 
\[
\cP\xrarrow{[T_1]}\cR\xrarrow{[T_2]}\cQ
\]
where $T_2\subseteq T$ and $T_1$ is uniquely determined by $T_2$ since $T=T_2\cup\mu_{T_2}(T_1)$ and $\mu_{T_1}$ is a bijection. Furthermore, there exists a morphism from $\cP\to\cR\to \cQ$ to \[
\cP\xrarrow{[T_1']}\cR'\xrarrow{[T_2']}\cQ
\]
 if and only if $T_2'\subseteq T_2$. And this morphism $\cR\to\cR'$ is unique since it is $\mu_{T_2'}^{-1}$ of the complement of $T_2'$ in $T_2$.
 
 The partitions $\cR$ corresponding to subsets $T_2$ of $T$ are also uniquely determined since they are given by merging the corresponding pairs of parts of $\cQ$ together. Pairs not in $T_2$ are not merged and thus lie in separate parts of $\cR$. So, different $T_2$ give different $\cR$.
 
 This description of $Fac[T]$ proves the lemma.
 \end{proof}

 We now consider first factors of cluster morphisms. Let $\{V\}$ be the partition of $V=\{v_1,v_2,\cdots,v_n\}$ into one part. All possible noncrossing partitions of $V$ into two parts are given by $\cP_{ij}=\{X_{ij}, Y_{ij}\}$ where
 \[
 	X_{ij}=\{v_{i+1},\cdots,v_j\},\quad 0\le i<j< n
 \]
 and {\color{blue}$Y_{ij}=V\backslash X_{ij}=\{v_1,v_2,\cdots,v_{i-1},v_{j+1},\cdots,v_n\}$ is the complement of $X_{ij}$ in $V$.}
 
 \begin{defn}
 Let $\cC(V)$ be the set of all morphisms of rank one:
 \[
 	[M_{ij}]: \{V\}\to \cP_{ij}, 0\le i<j<n
 \]
 given by $M_{ij}=Y_{ij}-X_{ij}$ and 
 \[
 	[\overline P_j]:\{V\}\to \cP_{0j},0<j<n
 \]
 given by $\overline P_j=X_{0j}-Y_{0j}$.
 \end{defn}
 
 \begin{rem}
 These correspond to the objects of the cluster category of type $A_{n-1}$ with straight orientation. Rank one morphisms are compatible if and only if the corresponding objects in the cluster category do not extend each other. Maximal pairwise compatible sets have $n-1$ elements and form what are called ``cluster tilting objects'' in the cluster category. Thus, by definition, they are given by a pairwise compatibility condition. See \cite{BMRRT}, \cite{CCS} for details.
 \end{rem}
 
 \begin{thm}\begin{enumerate}
 \item The set $\cC(V)$ is the set of all morphisms from $\{V\}$ to noncrossing partitions with two parts. 
 \item The morphism $[\overline P_k]$ is compatible with $[M_{ij}]$ for $k\notin (i,j]$ and with all other $[\overline P_\ell]$. 
 \item $[M_{ij}],[M_{k\ell}]$ are compatible if and only if they are noncrossing, i.e., the intervals $(i,j], (k,\ell]$ are either disjoint or one contains the other. 
 \item A collection of elements of the set $\cC(V)$ form the set of first factors of a unique cluster morphism $\{V\}\to\cP$ if and only if they are pairwise compatible.
 \end{enumerate}
 \end{thm}
 
 \begin{proof}
 These statements follow from the definitions of the terms.
 \end{proof}
 
 \begin{rem}\label{rem: cluster morphism nomenclature}
 The reason that morphisms in $\cN\cP(n)$ are called ``cluster morphisms'' is because they are given by partial clusters in the cluster category. In this paper they are described in terms of the ``c-vectors'' of the cluster. In \cite{IT13} a purely representation theoretic approach is given using cluster tilting objects in the cluster categories of hereditary abelian subcategories called ``wide subcategories'' \cite{Ingalls-Thomas} of $mod\text-\Lambda$ for any hereditary algebra $\Lambda$.
 \end{rem}
 
 \begin{cor}
 $\cN\cP(n)$ is a cubical category satisfying all the properties of Proposition \ref{prop: good properties of C}.
 \end{cor}
 
 \begin{proof}
 To prove that $\cN\cP(n)$ is cubical, it remains to prove Property (3) in the definition of cubical category. Since we have shown that the first factors are distinct, we need only show that the set of first factors determines the morphism $[T]:\cP\to\cQ$. But $T$ is a disjoint union of rooted trees, one for each part of $\cP$. So, we can deal with each part of $\cP$ separately and we are reduced to the theorem above.
 
 The theorem above shows that Condition (2) in Proposition \ref{prop: good properties of C} is satisfied. The other conditions have already been verified.
 \end{proof}
 
 By Proposition \ref{prop: good properties of C} we obtain the following.
  
 \begin{thm}
 $B\cN\cP(n)$ is locally $CAT(0)$ and thus a $K(\pi,1)$.
 \end{thm}

\section{Fundamental group}

In this section we will compute the fundamental group of $B\cN\cP(n)$.

 \begin{thm}\label{thm: fundamental group is picture group}
 The fundamental group of $B\cN\cP(n)$ is the group $G(A_{n-1})$ having the following presentation. The generators are $x_{ij}$ where $1\le i<j\le n$ with relations:
 \begin{enumerate}
 \item $[x_{ij},x_{jk}]=x_{ik}$ for all $1\le i<j<k\le n$.
 \item $[x_{ij},x_{k\ell}]=1$ if $x_{ij},x_{k\ell}$ are noncrossing in the sense that the closed intervals $[i,j],[k,\ell]$ are either disjoint or one is contained in the interior of the other.
 \end{enumerate}
 Here $[x,y]:=y^{-1}xyx^{-1}$.
 \end{thm}
 
 \subsection{Generators}
 
 For the basepoint of $B\cN\cP(n)$ we take the partition $\Omega$ of $V=\{v_1,\cdots,v_n\}$ into one point sets. Let $\cP_{ij}$ be the partition of giving by merging $v_i,v_j$, $i<j$, into one part $\{v_i,v_j\}$. Then we have two morphisms $[\beta_{ij}], [-\beta_{ij}]:\cP_{ij}\to\Omega$ given by $\beta_{ij}=v_j-v_i$ and $-\beta_{ij}=v_i-v_j$.

 Then $x_{ij}$ is defined to be the homotopy class of the loop $[-\beta_{ij}]^{-1}[\beta_{ij}]$ at $\Omega$ where we use the convention that paths are always composed left to right. To show that the $x_{ij}$ generate $\pi_1 B\cN\cP(n)$, we use the following equivalence relation on morphisms.
 
Given two morphisms $[T],[S]:\cS\to\Omega$, we say that $[T]\sim[S]$ if there is a sequence of morphisms $[T]=[T_0],[T_1],\cdots,[T_m]=[S]$ so that $[T_i],[T_{i+1}]$ share a common first factor $\cS\to\cR_i$ for each $i$.

\begin{lem}
Let $\cS$ be a noncrossing partition of rank $k\ge 2$ then any two morphisms $[T],[S]:\cS\to\Omega$ are equivalent. 
\end{lem}

\begin{proof}
Suppose first that $\cS$ has at least two parts $X,Y$ with more than one element. Then $[T],[S]$ are given by binary trees on each of these parts. Let $[R]$ be any  morphism which is equal to $[T]$ on $X$ and $[S]$ on $Y$. Then the part of $[T]$ on $X$ gives a common first factor for $[T],[R]$. So $[T]\sim[R]$. Similarly, $[R]\sim [S]$. So, $[T]\sim[S]$.

Now suppose that $\cS$ has only one part $W$ which is not a singleton. Then $W$ has $k+1$ elements and the forward link of $\cS$ is a triangulation of the sphere $S^{k-1}$ into a Catalan number of simplices. When $k\ge2$ this is connected and thus any two morphisms are equivalent.
\end{proof}

\begin{prop}
The fundamental group of $B\cN\cP(n)$ is generated by the loops $x_{ij}$.
\end{prop}

\begin{proof}
Choosing a morphism from every object to $\Omega$ we see that every loop at $\Omega$ is a composition of loops of the form $[T]^{-1}[S]$ where $[T],[S]$ are morphisms $\cS\to\Omega$. If $\cS$ has rank 1 then this loop is $x_{ij}$ or its inverse for some $i,j$. Otherwise, $[S]\sim[T]$ by the lemma which implies that $[T]^{-1}[S]$ is a composition of loops going through some object $\cR$ of smaller rank than $\cS$. The proposition follows by induction on that rank.
\end{proof}

\subsection{Relations}\label{ss: example Sijk}

Let $(i,j),(k,\ell)$ be two noncrossing pairs of numbers between 1 and $n$. Then, there is a noncrossing partition $\cS$ of rank 2  having $\{v_i,v_j\}$ and $\{v_k,v_\ell\}$ as two of its parts. There are exactly four morphisms $\cS\to\Omega$ giving four factorization squares (2-cubes). {\color{blue}Recall that $\Omega$ is the unique partition of rank 0, dividing a set into its individual elements.}
\[
\xymatrix{
\Omega & \cP_{ij}\ar[l]_{[-\beta_{ij}]}\ar[r]^{[\beta_{ij}]}
&\Omega\\
\cP_{k\ell}\ar[d]_{[-\beta_{k\ell}]}\ar[u]^{[\beta_{k\ell}]}
&\cS\ar[u]\ar[d]\ar[l]\ar[r]
&\cP_{k\ell}\ar[d]^{[-\beta_{k\ell}]}\ar[u]_{[\beta_{k\ell}]}\\
\Omega & \cP_{ij}\ar[l]^{[-\beta_{ij}]}\ar[r]_{[\beta_{ij}]}
&\Omega	}
\]
Together, these four squares give a homotopy $x_{ij}x_{k\ell}\simeq x_{k\ell}x_{ij}$ giving Relation (2) in {\color{blue}Theorem \ref{thm: fundamental group is picture group}.}

To obtain the other relation consider the partition $\cS_{ijk}$ of rank 2 having one part $\{v_i,v_j,v_k\}$ and the other parts all singletons. {\color{blue}Then there are exactly five morphisms $\cS_{ijk}\to \Omega$ given by the binary trees in Figure \ref{Fig (b): five cluster morphisms of rank 2}. Each morphism has two factorizations giving the diagram shown in Figure \ref{Fig: pentagon}.}

{
\begin{figure}[htbp]
\begin{center}
\begin{tikzpicture}
	\draw[thick] (-2,.5)node[left]{$i$}--+(1,-.5)node[right]{$j$} +(0,0)--+(2,.5)node[right]{$k$} +(.3,-.6)node{$-\beta_{ij}$} +(1,.6)node{$\beta_{ik}$} +(1,-1.5)node{$(a)$};
	\draw[thick] (1,0)node[left]{$i$}--+(1,.6)circle[radius=1pt]node[above]{$j$}--+(2,1)node[right]{$k$} +(.6,-.1)node{$\beta_{ij}$} +(1.65,.4)node{$\beta_{jk}$}+(1,-1)node{$(b)$};
	\draw[thick] (4,0)node[left]{$i$}--+(1,1)node[above]{$j$}--+(2,0)node[right]{$k$} +(.3,.8)node{$\beta_{ij}$} +(1.7,.8)node{$-\beta_{jk}$} +(1,-1)node{$(c)$};
	\draw[thick] (9,0)node[right]{$k$}--+(-1,.6)circle[radius=1pt]node[above]{$j$}--+(-2,1)node[left]{$i$} +(-1.6,0.4)node{$-\beta_{ij}$} +(-.35,.6)node{$-\beta_{jk}$} +(-1,-1)node{$(d)$};
	\draw[thick] (12,.5)node[right]{$k$}--+(-1,-.5)node[left]{$j$} +(0,0)--+(-2,.5)node[left]{$i$} +(-1,0.6)node{$-\beta_{ik}$} +(-.4,-.6)node{$\beta_{jk}$} +(-1,-1.5)node{$(e)$};
\end{tikzpicture}
\color{blue}
\caption{These are the binary trees giving the five cluster morphisms $ \cS_{ijk}\to \Omega$ using Definition \ref{def: binary tree notation}. Each of these cluster morphisms has two factorizations giving 5 squares in the category of noncrossing partitions. These five squares are labeled $(a),(b),(c),(d),(e)$ in Figure \ref{Fig: pentagon}. Binary trees (a) and (e) are explained more fully in Figure \ref{Fig: factorization of morphism of rank 2}. Figure \ref{Fig: composition of morphisms using trees} shows how these diagrams are used to compute composition of morphisms.}
\label{Fig (b): five cluster morphisms of rank 2}
\end{center}
\end{figure}
}

{
\begin{figure}[htbp]
\begin{center}
\[
\xymatrixrowsep{10pt}\xymatrixcolsep{10pt}
\xymatrix{
&& \Omega
 && \cP_{ij}\ar[rr]^{[-\beta_{ij}]}\ar[ll]_{[\beta_{ij}]}&&  \Omega
 \\
&\cP_{jk}\ar[ur]^{[\beta_{jk}]}\ar[dl]_{[-\beta_{jk}]} && (b)&&(a)
\\
\Omega
&(c)&&&\cS_{ijk}\ar[rr]^{[\beta_{jk}]}\ar[dd]^{[-\beta_{ij}]}\ar[uu]_{[\beta_{ik}]}\ar[lllu]_(.6){[\beta_{ij}]}\ar[llld]_(.4){[-\beta_{ik}]}&&\cP_{ik}\ar[uu]_{[\beta_{ik}]}\ar[dd]^{[-\beta_{ik}]}\\
&\cP_{ij}\ar[ul]^{[\beta_{ij}]}\ar[dr]_{[-\beta_{ij}]}&& (d)&&(e)
&&\\
&& \Omega &&\cP_{jk}\ar[ll]_{[-\beta_{jk}]}\ar[rr]^{[\beta_{jk}]}&&\Omega
	}
\]
\color{blue}
\caption{The five morphisms $\cS_{ijk}\to\Omega$ given by the binary trees in Figure \ref{Fig (b): five cluster morphisms of rank 2} give five factorization squares in $B\cN\cP(n)$ which fit into a pentagon as shown above.}
\label{Fig: pentagon}
\end{center}
\end{figure}
}

Starting from the lower left $\Omega$, the two paths going up to the upper left $\Omega$ are $x_{ij}x_{jk}\simeq x_{jk}x_{ik}x_{ij}$ giving Relation (1). We now need to show that there are no other relations. We use the fact that, for any CW complex with one 0-cell, the generators of $\pi_1$ are given by the 1-cells and the relations are given by the 2-cells.

\begin{thm}\label{thm: 4.4}
The classifying space $B\cN\cP(n)$ is an $n-1$ dimensional CW-complex having one cell $e(\cS)$ of dimension $k$ for every noncrossing partition $\cS$ of rank $k$. The $k$-cell $e(\cS)$ is the union of all factorization cubes of all morphisms $\cS\to\Omega$.
\end{thm}

Assuming the theorem, the 1-cells are all the loops
\[
	\Omega\xlarrow{[-\beta_{ij}]}\cP_{ij}\xrarrow{[\beta_{ij}]}\Omega
\]
The 2-cells are the squares and pentagons given above giving Relations (2), (1). The higher cells do not affect the fundamental group.

\begin{proof}
Given any noncrossing partition $\cS=\{X_1,\cdots,X_m\}$, a morphism $[T]:\cS\to\Omega$ is a product of morphism $[T_i]$ one for each part $X_i$ of $\cS$. The forward link $Lk_+(\cS)$ is a join of spheres
\[
	Lk_+(\cS)=\ast S^{k_i-2}
\]
where $k_i=|X_i|$. The union of factorization cubes for the morphisms $\cS\to\Omega$ is a product of disks of dimension $k_i-1$. The key point is to show that these cells meet lower dimensional cells only along their boundaries. However, this is easy. The interior of the cell $e(\cS)$ is the set of all points $x\in B\cN\cP(n)$ with the property that $\varphi(x)$ is a morphism with source $\cS$. The boundary of the cell consists of unions of those faces of factorization cubes of morphisms $\cS\to\Omega$ which do not include $\cS$ as vertex. These are morphisms having source $\cR$ of smaller rank than $\cS$ which occur in factorizations $\cS\to\cR\to\Omega$. Therefore, $B\cN\cP(n)$ is a CW complex as claimed.
\end{proof}

\section{Relation to cluster categories}

The purpose of this section is to explain the relationship between the category of noncrossing partitions as explained in detail in this paper and the category of \cite{HK} which we denote $\cH\cK$. The basic topological difference is that the category $\cH\cK$ has an initial object $0$ and is therefore contractible. The same for the opposite category $\cH\cK^{op}$ where $0$ is the terminal object. The category $\cN\cP(n)$ has a Catalan number $C_n=\frac1{n+1}\binom{2n}n$ of morphisms from $\{1,\cdots,n\}$ to the object $\Omega$ which corresponds to $0$. There does not appear to be a functor from one category to the other. However, there is a third category that is related to both. This is an extended version of the ``cluster morphism category'' from \cite{IT13}.

\subsection{Preliminaries on categories}

Suppose that $F:\cC\to \cD$ is a functor and let $X\in\cD$. Then the \emph{comma category} $F\!\downarrow\!X$ (\cite{MacLane98}) is defined to be the category of pairs $(Y,f)$ where $Y$ is an object in $\cC$ and $f:FY\to X$ is a morphism in $\cD$. A morphism $(Y,f)\to (Z,g)$ is a morphism $h:Y\to Z$ in $\cC$ so that $f=g\circ Fh:FY\to FZ\to X$. In the special case when $F$ is the identity functor $\cC\to \cC$, we use the notation $\cC\!\downarrow\!X$ for $id_\cC\!\downarrow\!X$. Dually, let $X\!\downarrow\! F$ denote the \emph{(left) comma category} with objects all pairs $(f,Y)$ where $Y\in\cC$ and $f:X\to FY$. A morphism $(f,Y)\to (g,Z)$ in $X\!\downarrow\! F$ is a morphism $h:Y\to Z$ in $\cC$ so that $g=Fh\circ f:X\to FY\to YZ$. The following is an easy exercise.

\begin{prop}\label{prop: FX is full subcategory of JFX}
Let $F:\cC\to\cD'$ be a functor from $\cC$ to a subcategory $\cD'$ of $\cD$ and let $J:\cD'\into \cD$ be the inclusion functor. Let $X\in\cD'$. Then $F\!\downarrow\!X$ is a full subcategory of $JF\!\downarrow\!X$. \qed
\end{prop}

Given two equivalent Krull-Schmidt categories $\cC,\cD$ we define an \emph{functorial bijection} $f:\Ind\,\cC\to \Ind\,\cD$ to be a bijection between the sets of isomorphism classes of indecomposable objects of $\cC,\cD$ which comes from an equivalence of categories $\cC\cong \cD$. Thus, for example, if $\cC,\cD$ have only finitely many indecomposable objects, then there are at most finitely many functorial bijections $\Ind\,\cC\to \Ind\,\cD$ although there may be infinitely many equivalences.

\subsection{Cluster morphism category}

Let $\cA$, $\cB$ be hereditary abelian categories which are equivalent to module categories of finite dimensional algebras over the same field, say ${K}$. Let $\cC_\cA,\cC_\cB$ be the corresponding cluster categories. (See \cite{BMRRT} for details.) A \emph{cluster morphism} $\cC_\cA\rightharpoonup \cC_\cB$ is defined to be a pair $([T],f)$ where 
\begin{enumerate}
\item $[T]$ is (the isomorphism class of) a partial cluster tilting object $T\in \cC_\cA$ and 
\item $f:\Ind(T^\perp)\to \Ind\,\cB$ is a functorial bijection. 
\end{enumerate}
Recall that $T^\perp$ is the full subcategory of the category $\cA$ of all modules $M$ so that $\Hom(|T|,M)=0=\Ext(|T|,M)$ where $|T|$ is the underlying module of $T$ (replacing shifted projective summands $P[1]$ with $|P[1]|=P$). Composition of morphisms uses the bijection between ordered clusters and signed exceptional sequences. See \cite{IT13}. As a special case, $\Hom(\cC_\cA,0)$ is the set of isomorphism classes of cluster-tilting objects of $\cC_\cA$. 

Let $\cG_{K}$ denote the category whose objects are hereditary abelian ${K}$-categories as discussed above and whose morphisms $\cA\to\cB$ are defined to be the cluster morphisms $\cC_\cA\rightharpoonup\cC_\cB$ as defined above.

In \cite{IT13} we consider the following related category. For a fixed object $\cA\in\cG_{K}$, let $\cG(\cA)$ denote the category of all finitely generated wide subcategories $\cW\subseteq\cA$. These are exactly the categories which occur as perpendicular categories of partial cluster tilting objects $T$ in $\cC_\cA$. A morphism $\cW_1\to \cW_2$ in $\cG(\cA)$ is defined to be an isomorphism class $[S]$ of a partial cluster tilting object $S$ in $\cW_1$ so that $S^\perp\cap \cW_1=\cW_2$. There is a functor $J:\cG(\cA)\to \cG_{K}$ given by the identity on objects $J\cW=\cW$ and on a morphism $[S]:\cW_1\to\cW_2$ we take $J[S]$ to be the pair $([S],f)$ where $f:\Ind(S^\perp\cap \cW_1)\to \Ind \cW_2$ is the identity mapping.

\begin{prop}
Let $\overline\cG_{K}$ denote the quotient category of $\cG_{K}$ with the same objects but with equivalence classes of morphisms where $([T],f)\sim([S],g)$ if $T^\perp=S^\perp$ and $f=g$. Let $F:\cG_{K}\to\overline\cG_{K}$ be the forgetful functor. Then $\cG(\cA)$ is equivalent to the comma category $\cA\!\downarrow\! F$.
\end{prop}

\begin{proof}
An object of $\cA\!\downarrow\! F$ is a pair $(f,\cB)$ where $f:\Ind \cW_0\cong \Ind\cB$ is a functorial bijection where $\cW_0$ is a finitely generated wide subcategory of $\cA$. If $(g,\cC)$ is another object of $\cA\!\downarrow\! F$ we have $g:\Ind\cW_1\cong \Ind\cC$. A morphism $(f,\cB)\to (g,\cC)$ is a cluster morphism $([T],h):\cB\to \cC$ so that the following diagram commutes.
\[
\xymatrix{
\Ind\cW_0\ar[r]^{f}& \Ind\cB \\
\Ind\cW_1\ar[u]^\subseteq \ar[r]\ar@/_2pc/[rr]^g & \Ind(T^\perp\cap\cB)\ar[u]^\subseteq \ar[r]^(.6)h & \Ind\cC
	}
\]
In other words, $\cW_1=\cW_0\cap S^\perp$ where $S=f^{-1}(T)$. An equivalence of categories $\cA\!\downarrow\! F\cong \cG(\cA)$ is given by sending $(f,\cB)$ to $\cW_0$ and $([T],h):(f,\cB)\to (g,\cC)$ to $[f^{-1}(T)]:\cW_0\to\cW_1$. The inverse $\cG(\cA)\cong\cA\!\downarrow\! F$ is given by taking $\cW$ to $(id,\cW)$.
\end{proof}

{\color{blue}Since $\cN\cP(n)$ is equivalent to the cluster morphism category of the path algebra of $A_{n-1}$ with straight orientation we conclude the following (for any field ${K}$).}

\begin{cor}\label{cor: NCP as cluster morphism category}
The category $\cN\cP(n+1)$ is equivalent to the comma category ${K} A_n\!\downarrow\! F$ where $F:\cG_{K}\to \overline\cG_{K}$ is the forgetful functor and $A_n$ denotes the quiver $1\ot 2\ot \cdots \ot n$. \qed
\end{cor}

\subsection{The category of Hubery-Krause}

We recall that the objects of the Hubery-Krause category $\cH\cK$ are pairs $(L,E)$ where $L\cong \ZZ^n$ (for $n$ variable) with an Euler form $\left<\cdot,\cdot\right>$ coming from some finite dimensional hereditary algebra $\Lambda$ over a field ${K}$ and $E\subset L$ is the set of dimension vectors of a complete \emph{exceptional sequence up to sign}. (Exceptional sequences are defined in \cite{CB}. All $2^n$ possible signs for an exceptional sequence are allowed in \cite{HK}. But in \cite{IT13} not all signs are allowed. E.g., there are only $n!C_{n+1}$ \emph{signed exceptional sequences} of type $A_n$ corresponding to the $n!$ permutations of the $C_{n+1}$ clusters.) A morphism $(L',E')\to (L,E)$ is a linear embedding $\varphi:L'\into L$ which is an isometry with respect to the Euler forms and which sends $E'$ to the set of dimension vectors of an exceptional sequence up to sign. The representation theoretic language can be removed from the definition since the dimension vectors of exceptional sequences in $L$ depends only on $E$ and the Euler form.

There is a (contravariant) functor $\cG_{K}\to \cH\cK^{op}$ given by sending $\cA$ to the pair $(K_0(\cA),E)$ where $E=\{[S_i]\}$ is the basis of $K_0(\cA)$ given by the simple modules. (See \cite{HK}, Lemma 5.3.) A morphism $([T],f):\cA\to \cB$ is sent to the monomorphism $f^\ast:K_0(\cB)\to K_0(\cA)$ induced by $f$. We denote this functor by $K_0$. Thus $K_0(\cA)=(K_0(\cA),E)$ with $E$ understood. Since $f$ is uniquely determined by $f^\ast$, we get the following.

\begin{lem}
The functor $K_0:\cG_{K}\to \cH\cK^{op}$ factors through the forgetful functor $F:\cG_{K}\to\overline\cG_{K}$ and the induced functor $\overline\cG_{K}\to \cH\cK^{op}$ is faithful (but not full).\qed
\end{lem}

An example of a morphism in $\cH\cK^{op}$ which is not in $\overline\cG_{K}$ is multiplication by $-1$ (corresponding to the shift $[1]$ in the derived category). By Proposition \ref{prop: FX is full subcategory of JFX} we get the following.

\begin{thm}
Let $\cA$ be an object of $\cG_{K}$. Then the opposite category of $\cG(\cA)$ is equivalent to a full subcategory of $K_0\!\downarrow\!K_0(\cA)$.\qed
\end{thm}

By Corollary \ref{cor: NCP as cluster morphism category}, we get the following special case.

\begin{cor}
The opposite category of $\cN\cP(n+1)$ is equivalent to a full subcategory of $K_0\!\downarrow\!K_0({K} A_n)$.\qed
\end{cor}

\begin{eg}
Let $\cA_2$ be the module category of ${K} A_2$ the path algebra of the quiver $1\ot 2$. Then $\cA_2$ has five wide subcategories: $0,\cA(S_1),\cA(S_2),\cA(P_2),\cA_2$. These correspond to the five noncrossing partitions of the set $\{0,1,2\}$: $(0)(1)(2), (01)(2), (0)(12), (1)(02), (012)$, resp. Up to isomorphism, these are the five objects of $\cN\cP(3)$ and there are 11 morphisms of rank one and 5 morphisms of rank two between these arranged as follows.
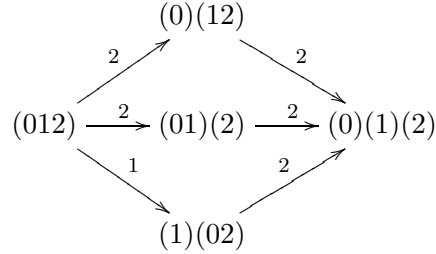
\[
\xymatrix{
& (0)(12)\ar[dr]^2\\
(012) \ar[r]^2\ar[ur]^2\ar[dr]^1& (01)(2)\ar[r]^2 & (0)(1)(2)\\
	& (1)(02) \ar[ur]^2
	}
\]
The labels on the arrows indicate the number of rank one morphisms. There are five rank 2 morphisms which are given in detail in subsection \ref{ss: example Sijk} with $i,j,k=0,1,2$ and labeled:
\[
	(a),(b),(c),(d),(e):\cS_{ijk}\to \Omega.
\]

Each rank 2 morphism has $2!=2$ possible factorizations given by the possible orderings of the set of edges. These give the $2!5=10$ signed exceptional sequences. There are 12 exceptional sequences up to sign. {\color{blue}They are $(\pm \beta_{01},\pm\beta_{02})$, $(\pm\beta_{02},\pm\beta_{12})$ and $(\pm\beta_{12},\pm\beta_{01})$.
}The two extraneous ones are: $(\pm\beta_{02},-\beta_{12})$. These are not signed exceptional sequences since $\beta_{12}$ does not correspond to a projective module. (The only negative objects in the cluster category are the shifted projective objects.) {\color{blue}This was also seen in Figure \ref{Fig: factorization of morphism of rank 2}. The only morphism allowed to precede $\pm\beta_{02}$ is $+\beta_{12}$ since $W=\{0,2\}$ covers $1$, so the edge must go from $1$ up to $W$ which is $+\beta_{12}$.}

In the category $\cH\cK$, we have $K_0({K} A_2)=(\ZZ^2,(\alpha_2,\alpha_1))$ with Euler form given by the Euler matrix $\mat{1 & 0\\-1&1}$ and $\alpha_i=e_i$ are the basis vectors. There are six morphisms $K_0({K})=(\ZZ,\alpha_1)\to K_0({K} A_2)$ given by sending $\alpha_1$ to the six roots $\pm \alpha_1,\pm \alpha_2,\pm \beta$ where $\beta=\alpha_1+\alpha_2$. Since $K_0({K} A_2)$ has six automorphisms, the comma category $K_0\!\downarrow\!K_0({K} A_2)$ has 13 objects (up to isomorphism): 
\begin{enumerate}
\item one object $(0,0)$ corresponding to the unique morphism $(0,\emptyset)\into K_0({K} A_2)$,
\item six objects $({K} A_1,(\beta_i)_\ast)$ where $\beta_i$ runs through the six roots of the root system, $(\beta_i)_\ast:K_0({K} A_1)\to K_0({K} A_2)$ being the isometric embedding sending $\alpha_1$ to $\beta_i$, 
\item  six objects $({K} A_2,\phi^i)$, $i=0,\cdots,5$ where $\phi\in\Aut(K_0({K} A_2))=\ZZ/6$ is a generator.
\end{enumerate}
The embedding of $\cN\cP(3)$ as a full subcategory of $K_0\!\downarrow\!K_0({K} A_2)$ sends the five objects of $\cN\cP(3)$ to:
\[
\begin{tabular}{cccc}
NCP & wide subcategory of $\cA_2$ & object of $K_0\!\downarrow\!K_0({K} A_2)$ & $T$\\
\hline
(0)(1)(2) & 0 & (0,0) & $T=T_1\oplus T_2$ (5 cases)\\
(0)(12) & $\cA(S_2)$ & $({K} A_1,(\alpha_2)_\ast)$ &  $T=P_1,P_1[1]$\\
(01)(2) & $\cA(S_1)$ & $({K} A_1,(\alpha_1)_\ast)$ &  $T=P_2,P_2[1]$\\
(1)(02) & $\cA(P_2)$ & $({K} A_1,(\beta)_\ast)$ &  $T=S_2$\\
(012) & $\cA_2$ & $({K} A_2,\phi^0=id)$ &  $T=0$\\
\end{tabular}
\]
The last column lists the morphisms from $\cA_2$ to the object in terms of representation theory. For example, there is only one morphism $(012)\to (1)(02)$ given by the cluster $T=S_2$. Since $S_2$ is not projective, there is no $S_2[1]$. In terms of the noncrossing partition, $(02)$ covers $(1)$, so the ordering is determined.

The category $K_0\!\downarrow\!K_0({K} A_2)$ has $6\times 7=42$ rank one morphisms. Of these, 11 lie in $\cN\cP(3)\cong \cG(\cA_2)$. For each $i$ there are two morphisms $0\to ({K} A_1,\beta_i)$ (of these the six that lie in $\cG(\cA_2)$ are the ones corresponding to the three positive roots) and five morphisms $({K} A_1,\beta_\ast)\to({K} A_2,\phi^i)$. Only the five where $i=0, \phi^0=id$ occurs in $\cG(\cA_2)$.

The fundamental groups of $\cN\cP(3)$ and $K_0\!\downarrow\!K_0({K} A_2)$ are:  
\[
	\pi_1(\cN\cP(3))=\left<x_1,x_2,x_3\,:\, x_2=[x_1,x_3]\right>=F_2
\]where $[x,y]:=y^{-1}xyx^{-1}$ and 
\[
\pi_1(K_0\!\downarrow\!K_0({K} A_2))=\left<x_1,\cdots,x_6\,:\, x_i=[x_{i-1},x_{i+1}], i=1,\cdots,6\right>
\]
The embedding of $\cN\cP(3)$ into $K_0\!\downarrow\!K_0({K} A_2)$ sends $x_i$ to $x_i$ for $i=1,2,3$.
\end{eg}

The relation between the categories of noncrossing partitions given in this paper and the one in \cite{HK} is apparently very complicated even in the smallest examples. However, the following diagram summarizes the connection.
\[
	\cN\cP\into \cG_{K}\onto \overline\cG_{K}\into \cH\cK^{op}
\]

\section*{Acknowledgements}

The author thanks his coauthors Moses Kim, Kent Orr, Gordana Todorov and Jerzy Weyman of the other papers on which this paper is based. He also thanks Ira Gessel and Olivier Bernardi for explaining the ballot numbers to him. Danny Ruberman was very helpful for using computer software to show that the groups $G(A_3)$ for different orientations of the quiver $A_3$ are nonisomorphic. (We use straight orientation in this paper.) Ruth Charney and her students, especially Matt Cordes, gave very nice talks about $CAT(0)$ spaces which were very helpful to prepare this paper. Also the author thanks Henning Krause for explaining his joint paper \cite{HK} to him and for his help with the bibliography. The author was supported by National Security Agency Grant \#H98230-13-1-0247 to investigate topological methods in representation theory. Version 5 (this version) is supported by Simons Foundation Grant \#686616. The author also thanks the anonymous referee for extensive comments and suggestions.

\end{document}